\documentclass[reqno,11pt]{amsart}

\numberwithin{equation}{section}

\usepackage[utf8]{inputenc}
\usepackage[english]{babel}

\usepackage{amsmath}
\usepackage{amsfonts,amssymb}
\usepackage{booktabs}
\usepackage{esint} 
\usepackage{amsthm}

\usepackage{nicefrac}

\usepackage{accents}
\usepackage{epsfig}
\usepackage{graphicx}
\usepackage{csquotes}

\usepackage{psfrag}
\usepackage{graphicx}
\usepackage{graphpap,latexsym,epsf}
\usepackage{color}
\usepackage{amssymb,mathrsfs,enumerate}
\usepackage[colorlinks, citecolor=citegreen, linkcolor=refred]
{hyperref}

\usepackage{caption}
\usepackage{subcaption}
\usepackage{float}

\usepackage{enumitem}

\usepackage[square,sort,comma,numbers]{natbib}

\makeatletter
\def\@tocline#1#2#3#4#5#6#7{\relax
  \ifnum #1>\c@tocdepth 
  \else
    \par \addpenalty\@secpenalty\addvspace{#2}%
    \begingroup \hyphenpenalty\@M
    \@ifempty{#4}{%
      \@tempdima\csname r@tocindent\number#1\endcsname\relax
    }{%
      \@tempdima#4\relax
    }%
    \parindent\z@ \leftskip#3\relax \advance\leftskip\@tempdima\relax
    \rightskip\@pnumwidth plus4em \parfillskip-\@pnumwidth
    #5\leavevmode\hskip-\@tempdima
      \ifcase #1
       \or\or \hskip 1em \or \hskip 2em \else \hskip 3em \fi%
      #6\nobreak\relax
    \dotfill\hbox to\@pnumwidth{\@tocpagenum{#7}}\par
    \nobreak
    \endgroup
  \fi}
\makeatother
\makeatletter
\newcommand{\xdashrightarrow}[2][]{\ext@arrow 0359\rightarrowfill@@{#1}{#2}}
\def\rightarrowfill@@{\arrowfill@@\relax\relbar\rightarrow}
\def\arrowfill@@#1#2#3#4{%
  $\m@th\thickmuskip0mu\medmuskip\thickmuskip\thinmuskip\thickmuskip
   \relax#4#1
   \xleaders\hbox{$#4#2$}\hfill
   #3$%
}
\makeatother
\usepackage{hyperref}

\newcommand{\R}{\mathbb{R}}
\newcommand{\RRR}{\mathbb{R}^{3{\times}3}}
\newcommand{\N}{\mathbb{N}}

\newcommand{\ep}{\varepsilon}

\newcommand{\SO}{{\rm SO}}
\newcommand{\Sym}{{\rm Sym}}

\newcommand{\tr}{{\rm tr}}
\newcommand{\sym}{{\rm sym}\,}

\newcommand{\rmd}{{\rm d}}
\newcommand{\LL}{{\rm L}}
\newcommand{\WW}{{\rm W}}
\newcommand{\Wiso}{\WW^{2,2}_{{\rm iso}}}
\newcommand{\trsp}{\scriptscriptstyle{\mathsf{T}}}
\mathchardef\emptyset="001F

\definecolor{vgreen}{rgb}{0.1,0.5,0.2}
\definecolor{viola}{RGB}{85,26,139}
\definecolor{citegreen}{rgb}{0,0.6,0}
\definecolor{refred}{rgb}{0.8,0,0}

\newtheorem{thm}{Theorem}[section]

\newtheorem{lemma}[thm]{Lemma}

\newtheorem{defin}[thm]{Definition}

\newtheorem{remarkr}[thm]{Remark}
\newenvironment{remark}
  {\pushQED{\qed}\remarkr}
  {\popQED\endremarkr}

\newtheorem{examplex}[thm]{Example}
\newenvironment{example}
  {\pushQED{\qed}\examplex}
  {\popQED\endexamplex}

\makeindex{}

\DeclareMathOperator*{\esssup}{ess\,sup}

\makeatletter
\renewcommand*\env@matrix[1][*\c@MaxMatrixCols c]{%
  \hskip -\arraycolsep
  \let\@ifnextchar\new@ifnextchar
  \array{#1}}
\makeatother

\newcommand*{\temp}{\multicolumn{1}{r|}{}}
\newcommand{\ubar}[1]{\underaccent{\bar}{#1}}

\begin{document}

\title[Heterogeneous elastic plates with in-plane modulation of the target curvature]
{Heterogeneous elastic plates with in-plane modulation of the target curvature and applications to thin gel sheets}

\author[V.~Agostiniani]{Virginia Agostiniani}
\address{SISSA, via Bonomea 265, 34136 Trieste - Italy} 
\email{vagostin@sissa.it}
\author[A.~Lucantonio]{Alessandro Lucantonio}
\address{SISSA, via Bonomea 265, 34136 Trieste - Italy}
\email{alucanto@sissa.it}
\author[D.~Lu\v ci\' c]{Danka Lu\v ci\' c}
\address{SISSA, via Bonomea 265, 34136 Trieste - Italy}
\email{dlucic@sissa.it}

\begin{abstract} 
We rigorously derive a Kirchhoff plate theory, 
\textit{via} $\Gamma$-convergence, from a three-di\-men\-sio\-nal model that describes the finite elasticity of 
	an elastically heterogeneous, thin sheet. 
The heterogeneity in the elastic properties of the material 
results in a spontaneous strain that depends on both the thickness 
	and the plane variables $x'$.
At the same time,
the spontaneous strain is 
$h$-close to the identity, where $h$ is the small parameter quantifying the thickness. The 2D 
Kirchhoff 
limiting model
is constrained to the set of isometric immersions
of the mid-plane of the plate into $\R^3$, 
with  
a corresponding energy
that penalizes
deviations of the curvature
tensor associated with a deformation
from a $x'$-dependent 
target curvature tensor.
A discussion on the 2D minimizers 
is provided in the case where the target curvature tensor is piecewise constant. Finally, we apply 
the derived
plate theory to the modeling of swelling-induced shape changes in heterogeneous 
thin gel sheets.
\end{abstract}

\makeatletter

\date{\today}


\maketitle



\tableofcontents

\section[Introduction]{Introduction}
Plants \cite{Armon,Dawson} and other natural systems \cite{Arroyo1,Arroyo2} are able to perform complex shape changes that produce curved configurations, often starting from flat initial states. These shape changes usually involve thin structures, such as membranes, plates or shells, and exploit some internal  activation or the responsiveness of the material to non-mechanical external triggers, such as changes in humidity. By mimicking natural behaviors and architectures, synthetic, polymer-based thin sheets have been fabricated that can spontaneously deform in response to non-mechanical stimuli. In particular, in these systems curvature arises from heterogeneous in-plane \cite{deHaan2012,Kle_Efr_Sha_2007,Kim,Wu2013,Pezzulla2015} or through-the-thickness strains \cite{ADS,ADS17,Pezzulla2016,hybrid,Sawa2011,Stoychev2012}, which are induced by heterogeneous material properties, including variable anisotropy. Thus, to study and control the emerging shapes, 
the derivation of plate theories for materials with heterogeneous response to external stimuli has become as a topic of interest in both the mathematical and the physical literature \cite{BBN,Aha_Sha_Kup,BLS,sharon_mechanics_2010,Lucantonio2016,Mostajeran}. 

\smallskip

In this framework,
we wish to
contribute by drawing attention to
some new plate theories and corresponding mathematical problems in dimension reduction inspired by shape morphing applications involving polymer gels. Specifically, in these applications one wants to program the material properties of
a thin gel sheet 
$\Omega_h=\omega\times(-h/2,h/2)$,
where $\omega\subseteq\mathbb{R}^2$ is the mid-plane
and $0<h\ll1$ 
	is the thickness, in order
to endow it with 
	a controlled curvature that is
realizable, 
upon swelling,
at the minimum 
energy cost.
Practically, curvature of the sheet
can be 
obtained
by imprinting
a heterogeneous density $N$ of polymer chains, 
	which corresponds to a heterogeneous shear modulus of the polymer network.
For concreteness, we consider the case
where $N$ is a 
(small) perturbation of order $h$ of the average value $\overline N$, that is,

\begin{equation}
\label{heterog_density}
N
\,=\,
\overline N^h(z)
\,=\,
\overline N
+
hg\left(z',\frac{z_3}h\right),
\qquad\qquad z=(z',z_3)\in\Omega_h,
\end{equation}
for some bounded function $g:\omega\times(-1/2,1/2)\to\R$.
Referring to the classical Flory-Rehner model 
	 \cite{Doi09} 
for isotropic polymer gels, 
we obtain 
	as a consequence of the above assumption on $N$
that the free energy density 
$\overline W^h$
associated with the system is minimized at
\begin{equation}
\label{freeen_minim}
\big(\alpha+hb(z',z_3/h)\big)\SO(3),
\qquad b=\Theta g,
\end{equation}
where the (dimensionless) constants $\alpha$ and 
$\Theta$ 
	are functions of
the material parameters appearing in the 
expression of $\overline W^h$, 
including $\overline{N}$.
We refer the reader to \eqref{flory_rehner}
for the explicit expression of $\overline W^h$
and to the whole Section \ref{gel_sheets}
for more details on this
3D model.
We recall that $\alpha$ corresponds to 
the free-swelling stretch of a homogeneous gel
(i.e. $g=0$ in \eqref{heterog_density})
with respect to its dry state, and is hence greater than one.

Intuitively, the connection between the density of polymer chains in \eqref{heterog_density} and the energy minimizers in \eqref{freeen_minim} offers a way to program minimum-energy strain fields that, as we will see in the following, induce a \textit{target} curvature for the system.
This mechanism of generation of curvature through heterogeneous elastic properties is poorly explored in the mathematical literature of active or ``pre-strained'' materials, and thus constitutes a novel ingredient of our theory. Tipically, these materials are modeled by 3D energy densities of the form (see, for example, \cite{Sch072})
\begin{equation}
\label{prestretched_form}
\overline W^h(z,F)
\,=\,
W\big(FV^h(z',z_3/h)\big),
\end{equation}
for a certain (frame indifferent) homogeneous energy density $W$ minimized at $\SO(3)$, where the ``pre-stretch'' $V^h$ is generally a smooth, 
invertible tensor field 
that represents
(the inverse of) an active stretch, growth, plasticity or other inelastic phenomena. 
In these models, $V^h$ plays the role of a parameter that is externally controlled, without any dependence on the elastic properties of the system (or on other parameters of the energy). Instead, in models based on the Flory-Rehner energy and on the relations \eqref{heterog_density}-\eqref{freeen_minim}, there is an intimate connection between material parameters and minimum energy deformations.

Another interesting feature of the Flory-Rehner energy and, correspondingly, of the family of energy densities we consider, is that they are \textit{not} representable in the pre-stretch form \eqref{prestretched_form}, with $V^h=\big(1+hb/\alpha\big)^{-1}\mathbb I_3$ 
and with energy density $W$ minimized at $\alpha\SO(3)$ (see Remark \ref{gel_motivation}). This feature depends on the different structure of such an energy with respect to the models based on the representation \eqref{prestretched_form}, which originates from physical considerations. 
More precisely,
in the latter models $\overline W^h$ has the physical meaning of purely elastic energy,
while in models for polymer gels $\overline W^h$ is the sum of two energy contributions (elastic and mixing energies) that concurrently define the energy minimum, but none of them is separately minimized at $\SO(3)(V^h)^{-1}$. 
However, the corresponding rescaled
densities converge
uniformly to some homogeneous density $W$.

\smallskip

Motivated by 
the above observations
and discarding 
	for the moment
the 
scalar parameter $\alpha$, which can be accounted for
by a simple change of variable, 
we thus consider the slightly more general setting of
a material characterized
by a \emph{spontaneous stretch
distribution} $\overline U^h$ of the form
\begin{equation}
\label{spont_stretch}
\overline U^h(z)
\,=\,
\mathbb I_3+
h\,B\left(z',\frac{z_3}h\right),
\end{equation}
where $B:\omega\times(-1/2,1/2)\to\Sym(3)$
is a given (bounded) strain distribution.
  
The term 
``spontaneous''
for the distribution
$\overline U^h$ 
	refers to the tendency of
the system to deform, at each point $z$, 
according to a deformation whose gradient
coincides with $\overline U^h(z)$, 
in order to attain the energy minimum pointwise. However, generally
there is no (orientation-preserving) deformation
defined globally in $\Omega_h$,
whose gradient coincides with $\overline U^h$
in the whole of $\Omega_h$.
Equivalently, in the words of Mechanics,
$\overline U^h$ is not kinematically compatible,
or, in the words of Riemannian geometry,
the Riemann curvature tensor associated with $\overline U^h$
	does not vanish
identically throughout $\Omega_h$.

\smallskip

It is now appropriate to notice that the 
3D setting just described can be seen
as a generalization of the setting considered in
\cite{Sch072} (see also \cite{Sch071}),
where the pre-stretch is of the same form
as in \eqref{spont_stretch}, 
except that the $z'$-dependence is not
considered. At the same time, the relevant case where
the pre-stretch in \eqref{prestretched_form} is 
	only
$z'$-dependent
has been addressed in \cite{LP11} and \cite{BLS} 
and has given rise to the fortunate route of
the mathematical treatment of the 
``non-Euclidean plate theories''
(see also \cite{Kupf14}),
introduced from a physical and mechanical 
view point by the pioneering work of Sharon
and coauthors in \cite{ESK} and \cite{Kle_Efr_Sha_2007}. All in all, our 
	theory stands
between   those of
	\cite{Sch072}, on one hand, and of 
	\cite{LP11} and \cite{BLS}, on the other hand,  and, to the best of our knowledge,
represents the first attempt to considering
Kirchhoff plate theories 
originated by 3D  
	energies
 characterized 
by pre-stretches or spontaneous stretches
which are  
	heterogeneous in plane
 as well as along the thickness.
Pre-stretches 
of the form \eqref{spont_stretch} 
have been very recently treated in \cite{CRS17} and \cite{Kohn16}
to derive corresponding rod models with misfit. 
Moreover, similar prestretches have been
considered in \cite{LOP} to obtain 2D models
in the case of scaling orders
higher than the Kirchhoff one.

\smallskip

The central result of this paper is the derivation of a Kirchhoff plate theory from the 3D model outlined above.
With abuse of notation, we again denote by 
$\overline W^h$ 
the energy density 
associated with this system, 
which is minimized, for every $z\in\Omega_h$, 
at $\SO(3)\overline U^h(z)$.
Hence, the total free energy associated with a
deformation $v:\Omega_h\to\R^3$ is
\[
\overline{\mathcal E}^h(v)
\,=\,
\int\limits_{\Omega_h}
\overline W^h\big(z,\nabla v(z)\big)\,\rmd z.
\]
 
	Then, the same arguments as in \cite{Sch072}
	(which are in turn 
	a slight variant of those 
	employed in the seminal work \cite{FJM02})
	can be used to find
	the corresponding limiting Kirchhoff plate model,
 under the assumption that 
\begin{equation}
\label{b_check}
{\rm curl}\big({\rm curl}\,\check D\big)
\,=\,
0,\qquad\quad\mbox{with}\qquad
\check D(z')
:=
\!\!\int\limits_{\nicefrac{-1}{2}}^{\nicefrac{1}{2}}
\!\!\check{B}(z',t)\,\rmd t
\qquad\mbox{for a.e. }z'\in\omega, 
\end{equation}
where $\check B:\omega\to\Sym(2)$ is obtained from the spontaneous strain distribution $B$ appearing in \eqref{spont_stretch}
by omitting the third row and the third column.
Condition \eqref{b_check} deserves some comments. It
guarantees that $\check D$ 
is a symmetrized gradient,
and in turn allows for the construction of a standard
ansatz for the recovery sequence.
When instead condition \eqref{b_check} is violated,
usual arguments such as local modifications 
or perturbation arguments seem insufficient to prove 
the same $\Gamma$-limit.
In fact, we believe that
the general $\Gamma$-limit has to include a nonlocal term,  
	which can be interpreted
	as a ``first order stretching term''.
  To conclude the comments on condition \eqref{b_check},
let us add a trivial but important observation:  the difficulties one encounter in removing 
	the compatibility assumption on the matrix field $\check D$
	do not originate from
	the dependence of the spontaneous strain
	on the thickness variable, since they
	persist even in the case where such a dependence is absent.

The Kirchhoff model resulting from the dimension reduction is
 governed by the 
energy functional
\begin{equation}\label{gamma_lim}
\mathcal{E}^0(y)\,=\,\frac{1}{24}\int_{\omega}Q_2\left({\rm A}_{y}(z')-\overline{A}(z')\right)\,\rmd z'+{\rm ad.t.},
\end{equation}
on each $\WW^{2,2}$-isometry $y$, where {\rm ad.t.} stays for  
``additional terms"	
 not depending on $y$.  
 
In the above expression,
the quadratic form $Q_2$ is defined via a standard
relaxation of the second differential of the limiting
density $W$ at $\mathbb I_3$
(see formulas \eqref{Q3} and \eqref{Q2}),
the symbol ${\rm A}_{y}$ stands for the pull-back  
of the second fundamental form associated with $y(\omega)$
(see \eqref{2ff}),
and the \emph{target curvature tensor} $\overline A$ 
is defined as 
\begin{equation*}
\overline{A}(z')
\,=\,
12\int_{\nicefrac{-1}{2}}^{\nicefrac{1}{2}}t\check{B}(z',t)\,\rmd t,
\qquad\quad\mbox{ for a.e. } z'\in \omega.
\end{equation*}
 
	It is readily seen that in the case
	where the prestretch depends only
	on the thickness variable, a constant
	target cuvature $\overline A$ is produced.
	For polymer gels, this expression makes the anticipated connection between heterogeneous density $N$ of polymer chains (encoded by $\check{B}$) and curvature more evident, even if not fully explicit. Further, under some approximations or using numerical methods,  such a relation can be made explicit and thus can be actually employed in the design of shape morphing gel plates. In general, our derivation, which relies on an accurate description of the 3D swelling energy, offers an advantage over  theories based on purely elastic energies with ``pre-stretch'', where such a connection must be plugged in artificially.

It is worth mentioning that
beam theories derived 
from 2D energies of the form \eqref{gamma_lim},
in the limit as $\ep\to0$ when 
$\omega=(-\ell/2,\ell/2)\times(\ep/2\times\ep/2)$,
can be found in \cite{ADK16} for the case $\overline A$
constant and in \cite{FHMP16}
in the case $\overline A=\overline A(x_1)$.
To use a common terminology, these 1D theories may describe
\emph{narrow ribbons} of soft active materials.  

\smallskip

To give some insight on the minimizers of the derived 2D model \eqref{gamma_lim}, we focus on the case where the 
spontaneous strain $B$
is an odd function of the thickness variable
(which trivially fulfills condition \eqref{b_check}),
being at the same time a piecewise constant
function of the planar variable.
This case leads in turn to a 
piecewise constant target curvature tensor.
In Section \ref{sec_min},
we recall that in the case where
$\overline A$ is constant, then
a minimizer of the 2D energy 
$\mathcal E^0$ actually minimizes
the integrand function pointwise,
and the corresponding deformed 
configuration is a piece of
cylindrical surface
(see Lemma \ref{min_A_const} and the discussion preceding it).
In the case of a piecewise constant 
$\overline A$, 
some conditions (specified in Theorem \ref{general_cond}) under which cylindrical surfaces can be patched together 
resulting into an isometry must be fulfilled for the pointwise minimizer to exist. When these conditions hold, an example of minimum energy configuration, a patchwork of cylindrical surfaces, is sketched in Figure \ref{fig: folding}.
 
\begin{figure}[t]
\includegraphics[scale=0.3]{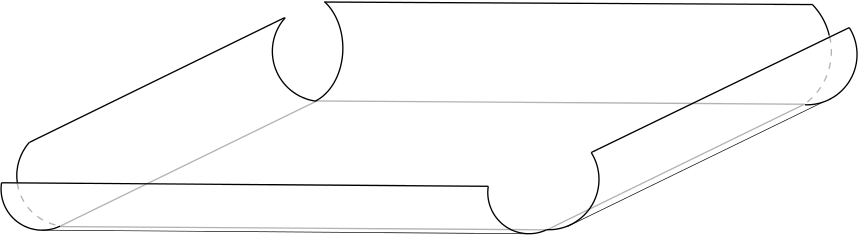}
\caption{
An example of a 2D minimum energy configuration.}
\label{fig: folding}
\end{figure}
\smallskip

The paper is structured as follows:
we deal with the theoretical
results
concerning dimension reduction
in Sections \ref{en_density}
and \ref{sec_min},
and then we apply them
to the case of thin gel sheets, 
in Section \ref{gel_sheets}.
In the final section, we draw some conclusions.
\smallskip

We end
this section by introducing some general 
notation which will be used throughout the paper.


\subsection[Notation]{Notation}
\label{notation}
For fixed $n\in \mathbb{N}$ we will denote by 
\begin{itemize}
\item  $\mathbb{R}^{n\times n}$ the vector space of real $n\times n$ matrices
and by $\mathbb{I}_n\in \mathbb{R}^{n\times n}$ the identity matrix,
\item $\Sym(n):=\{M\in \mathbb{R}^{n\times n}:M^{\trsp}=M\}$ the vector space of symmetric matrices, where by $M^{\trsp}\in \mathbb{R}^{n\times n}$ we denote the transpose of the matrix $M\in \mathbb{R}^{n\times n}$, 
\item ${\rm Skew}(n):=\{M\in \mathbb{R}^{n\times n}:M^{\trsp}=-M\}$ the set of skew-symmetric matrices,
\item $\SO(n):=\{M\in \mathbb{R}^{n\times n}: M^{\trsp}M=\mathbb{I}_n, {\rm det}(M)=1\}$ the set of all rotations of $\mathbb{R}^n$,
\item ${\rm Orth}(n):=\{M\in \mathbb{R}^{n\times n}: M^{\trsp}M=\mathbb{I}_n\}$ the set of all orthogonal transformations of $\mathbb{R}^n$,
\item ${\rm Trs}(n):=\{T_{v}:=\cdot+v: v\in \mathbb{R}^n\}$ the set of all translations in $\mathbb{R}^n$. Sometimes, to distinguish between translations in $\R^2$ and $\R^3$, we will denote by $\tau_v$ the elements of ${\rm Trs}(2)$, 
\item $M_{{\rm sym}}:=\frac{M+M^{\trsp}}2$ the symmetric part of the matrix $M\in \mathbb{R}^{n\times n}$,
\item $\tr\, M$ the trace of the matrix $M$ and $\tr^2 M:=(\tr M)^2$,
\item $|M|:=\sqrt{\sum_{i,j=1}^n|m_{ij}|^2}=\sqrt{{\rm tr}(M^{\trsp}M)}$, Frobenius norm of a matrix $M=[m_{ij}]_{i,j=1}^n\in \mathbb{R}^{n\times n}$,
\item $\mathcal{L}^n$ the $n$-dimensional Lebesgue measure,
\item $\mathcal{H}^n$ the $n$-dimensional Hausdorff measure.
\end{itemize}
\noindent
Furthermore, we give the following definitions:
\begin{itemize}
\item  $\check{F}\in \mathbb{R}^{2\times 2}$ is the $2\times 2$ submatrix of $F\in \mathbb{R}^{3\times 3}$ obtained by omitting the last row and the last column of $F$,
\item  given $G\in \mathbb{R}^{2\times 2}$, the matrix $\hat{G}\in \mathbb{R}^{3\times 3}$ associated to $G$ is defined as
$$\hat{G}=
\left(
\begin{array}{c|c}
G & \begin{array}{c}
0\\
0
\end{array} \\
\hline
0\quad 0 & 0
\end{array}
\right).
$$
\end{itemize}

We denote by $\{\mathsf{e}_1, \mathsf{e}_2\}$ the standard basis of $\mathbb{R}^2$ and by $\{\mathsf{f}_1, \mathsf{f}_2, \mathsf{f}_3\}$ the standard basis of $\mathbb{R}^3$. An open connected subset of $\mathbb{R}^2$ will be called \emph{domain}. Sometimes, for the sake of brevity, an open subset of $\mathbb{R}^2$ with Lischitz boundary will be called a \emph{Lipschitz subset} of $\mathbb{R}^2$. The closure of a set $S\subseteq\mathbb{R}^2$ is denoted by $\overline{S}$ or by ${\rm cl}(S)$.
\section[Three-dimensional model and derivation of the
corresponding 
Kirchhoff plate model]
{Three-dimensional model and derivation of the \\ corresponding 
Kirchhoff plate model}
\label{en_density}
Throughout the paper $\omega\subseteq \mathbb{R}^2$ will be a simply-connected, bounded domain with Lipschitz bo\-un\-da\-ry satisfying the following condition:
\begin{equation}\label{dom_cond}
\begin{aligned}
&\mbox{there exists a closed subset } \Sigma\subset \partial \omega\mbox{ with }  \mathcal{H}^1(\Sigma)=0 \mbox{ such that}\\
&\mbox{the outer unit normal exists and is continuous on } \partial \omega\setminus \Sigma. 
\end{aligned}
\end{equation}
The requirement that $\omega$ is 
a simply-connected domain has to do with the ``compatibility" condition of Theorem \ref{Saint_comp} below, which is imposed on 
the tensor-valued map $D_{\rm min}$ defined by \eqref{Bh} and \eqref{Dmin}.
The condition \eqref{dom_cond} is a standard requirement on the domain in order to have some 
density results for the space of $\WW^{2,2}$-isometric immersions 
of $\omega$ into $\mathbb{R}^{3}$
(see, e.g., \cite{Hor111} and \cite{Hor112}).

We are interested in a thin sheet 
$\Omega_h:=\omega\times(-h/2,h/2)$,
with  $0<h\ll 1$,
of a material characterized
by a \emph{spontaneous stretch}
given at each point of $\Omega_h$
in the form 
$\overline U^h(z)=
\mathbb I_3+h B\big(z',\frac{z_3}{h}\big)$,
for a suitable \emph{spontaneous strain}
$B\in\LL^{\infty}
(\Omega,\Sym(3))$.
The stretch $\overline U^h$ being 
\emph{spontaneous} for the material 
is modeled by introducing a energy density
whose minimum state
is precisely $\overline U^h(z)$
at each point $z$,
modulo superposed rigid
body rotations.
We denote by $U^h$
the spontaneous stretch given in  
terms of the rescaled variable
$x\in\Omega:=\Omega_1$.
Namely, $U^h(x)=\overline U^h(x',hx_3)$
so that
$U^h=\mathbb I_3+hB$. 

More in general, we consider a family
$\mathcal{B}=\{B^h\}_{h\geq 0}$ 
of spontaneous strains
such that
\begin{equation}\label{Bh}
B^h\to B^0=: B\qquad\mbox{ in }\ \LL^{\infty}(\Omega,\Sym(3)),\mbox{ as }h\to 0,
\end{equation}
the corresponding
family $\{U^h\}_{h\geq 0}$ 
of spontaneous stretches 
defined as
\begin{equation}\label{spont_strain}
U^h(x):=\mathbb{I}_3+hB^h(x)\quad\mbox{ for a.e.\ }x\in \Omega\mbox{ and for every }h\geq 0,
\end{equation}
and the associated family $\{W^h\}_{h>0}$ of (rescaled)
energy density functions
$W^h:\Omega\times \mathbb{R}^{3\times 3}
\to [0,+\infty]$, 
which we suppose to
be Borel functions satisfying 
the following properties:
\begin{itemize}
\item [(i)] for a.e.\ $x\in \Omega$, the map $W^h(x,\cdot)$ is frame indifferent, i.e.\
\[
W^h(x,F)=W^h(x, RF )\quad\mbox{ for every }F\in \mathbb{R}^{3\times 3}\mbox{ and every }R\in \SO(3);
\]
\item [(ii)]for a.e.\ $x\in \Omega$, 
$W^h(x,\cdot)$ is minimized precisely at  
$\SO(3)U^h(x)$;
\smallskip
\item [(iii)] 
there exists an open neighbourhood $\mathcal U$ of $\SO(3)$ 
and
$W\in C^2(\overline{\mathcal U})$
such that
\begin{equation}\label{cvg_wh}
\esssup_{x\in \Omega}\big\Vert W^h(x,\cdot)-W\big\Vert_{C^2(\mathcal{\overline{U}})}\to 0,\qquad\mbox{ as }h\to 0;
\end{equation}
\item [(iv)] there exists a constant $C>0$,
independent of $h$, 
such that for a.e.\ $x\in \Omega$ 
it holds that
\begin{equation}\label{gr_cond_wh} 
W^h(x,F)\geq C{\rm dist}^2\big(F, \SO(3)U^h(x)\big),
\qquad\mbox{for every } F\in \mathbb{R}^{3\times 3}.
\end{equation}
\end{itemize}

The most interesting scenarios occur when
the \emph{Cauchy-Green} distribution $C^h$
associated with the spontaneous stretch distribution $U^h$ 
$-$ namely, $C^h(x):=(U^h(x))^2$ $-$
is not kinematically compatible, i.e.\ there is no 
orientation-preserving deformation $v^h:\Omega\to \mathbb{R}^3$
such that $(\nabla v^h)^{\trsp}\nabla v^h=C^h$ in $\Omega$.
We also recall that, since $C^h(x)$ is a positive definite
symmetric matrix, 
the distribution $C^h$ can be interpreted as 
a metric on $\Omega$ and that, in this framework,
the kinematic compatibility of $C^h$
is equivalent to the condition that the
Riemann curvature tensor associated with
$C^h$ vanishes identically in $\Omega$
(see \cite{Ciarlet05} and \cite{LP11}).

\begin{defin}[Admissible family of free-energy densities]\label{adm_en_dens} Given 
$\mathcal{B}=\{B^h\}_{h\geq0}$
satisfying \eqref{Bh}
and the associated family
$\{U^h\}_{h\geq 0}$ defined in \eqref{spont_strain}, 
we call $\mathcal{B}\text{-}$\emph{admissible}
a family $\{W^h\}_{h>0}$ 
of Borel functions from  $\Omega\times \mathbb{R}^{3\times 3}$ to $[0,+\infty]$ 
fulfilling ${\rm (i)\text{-}(iv)}$.
\end{defin}

Given a $\mathcal{B}\text{-}$admissible 
family $\{W^h\}_{h> 0}$ of free-energy densities, with associated 
limiting density function $W$,
using a standard notation we define
the following quadratic form:
\begin{equation}\label{Q3}
Q_3(F):=D^2W(\mathbb{I}_3)[F,F],
\qquad \mbox{ for every }F\in \mathbb{R}^{3\times 3}.
\end{equation}
Moreover,
for every $G\in\mathbb{R}^{2\times 2}$,
we set
\begin{equation}\label{Q2}
Q_2(G):=\min_{d\in\mathbb{R}^3}Q_3\big(\hat{G}+d\otimes \mathsf{f}_3\big),
\end{equation}
referring to
Subsection \ref{notation}
for the notation $\hat G$.
Observe that
the limiting density $W$ inherits
properties (i), (ii) and (iv)
from convergence \eqref{cvg_wh}.
From this fact one can deduce that
$Q_2$ is indeed a quadratic form 
and that $Q_{k}$, for $k=2,3$,
has the following properties: 
\begin{itemize}
\item $Q_{k}$ is positive semi-definite on $\mathbb{R}^{k\times k}$ and positive definite when restricted to $\Sym(k)$,
\item $Q_{k}(F)=0$ for every $F\in {\rm Skew}(k)$,
\item $Q_{k}$ is strictly convex on $\Sym(k)$.
\end{itemize}
The proof of some of the listed properties can be found for instance in \cite{BLS, FJM02}.
We also refer to  
\cite[Proposition 11.9]{DalMaso} 
for a useful characterization of quadratic forms. 

Our limiting 2D model 
will be related to the 
2D density function
 $\overline{Q}_2:\omega\times\mathbb{R}^{2\times 2}\to [0,+\infty)$ 
defined as 
\begin{equation*}
\overline{Q}_2(x',G):=\min_{D\in \mathbb{R}^{2\times 2}}\int_{\nicefrac{-1}{2}}^{\nicefrac{1}{2}}Q_2\big(D+tG-\check{B}(x',t)\big)\,\rmd t,
\end{equation*}
for a.e.\ $x'\in \omega$ and every $G\in \mathbb{R}^{2\times 2}$,  where $\check B$ is related to the 3D model through \eqref{Bh}, using the notation introduced in Subsection \ref{notation}. 
Since $Q_2$ does not depend on the skew-symmetric part of 
its argument,
 
we can think of $\overline Q_2$ to be defined only on $\omega\times \Sym(2)$ as
\begin{equation}\label{Qbar}
\overline{Q}_2(x',G)=\min_{D\in \Sym(2)}\int_{\nicefrac{-1}{2}}^{\nicefrac{1}{2}}Q_2\big(D+tG-\check{B}(x',t)\big)\,\rmd t.
\end{equation}
 
  This minimum problem
can be solved explicitly,
as stated by the following lemma.

\begin{lemma}\label{Qbarmin}For a.e.\ $x'\in \omega$ and every $G\in \Sym(2)$,  the minimizer in \eqref{Qbar} is unique and coincides with 
\begin{equation}\label{Dmin}
D_{\min}(x')\,:=\,\int_{\nicefrac{-1}{2}}^{\nicefrac{1}{2}}\check{B}(x',t)dt.
\end{equation}
 
In other words, we have that
\begin{equation}
\label{Q_2_prima_forma}
\overline{Q}_2(x',G)=\int_{\nicefrac{-1}{2}}^{\nicefrac{1}{2}}Q_2\left(\int_{\nicefrac{-1}{2}}^{\nicefrac{1}{2}}\check{B}(x',s)\,\rmd s+tG-\check{B}(x',t)\right)\,\rmd t
\end{equation}
for a.e.\ $x'\in \omega$ and every $G\in \Sym(2)$.
\end{lemma}
\begin{proof}
 
By using the bilinear form associated with $Q_2$ it is easy to see that for a.e.\ $x'\in \omega$ and every $G\in \Sym(2)$ it holds
\[
\begin{aligned}
&\min_{D\in \Sym(2)} \int_{\nicefrac{-1}{2}}^{\nicefrac{1}{2}}Q_2(D+tG-\check{B}(x',t))\,\rmd t\\
=&\int_{\nicefrac{-1}{2}}^{\nicefrac{1}{2}}Q_2\big(tG-\check{B}(x',t)\big)\,\rmd t-Q_2\big(D_{\min}(x')\big)+\min_{D\in \Sym(2)} Q_2\big(D-D_{\min}(x')\big).
\end{aligned}
\]
From this equality, the thesis trivially follows.
 
\end{proof}

Note that
$D_{\rm min}$,
which is in principle dependent on
$G$ from its definition,
turns out to be independent
of $G$ in the end.
This is not the case when,
e.g., the limiting density function $W$
depends explicitly on $x_3$,
not just through its spontaneous
stretch, see \cite{Sch072}.
Note also from hypothesis 
\eqref{Bh} that 
$D_{\rm min}\in 
\LL^{\infty}(\omega, \Sym(2))$. 
Finally, observe  for future reference that from 
\eqref{Q_2_prima_forma} 
one can  rewrite  $\overline{Q}_2$  in the more explicit  form
\begin{multline}
\label{limit_Q}
\overline{Q}_2(x',G)
\,=\,
\frac{1}{12}Q_2\left(G-12\int_{\nicefrac{-1}{2}}^{\nicefrac{1}{2}}t\check{B}(x',t)\,\rmd t\right)+\int_{\nicefrac{-1}{2}}^{\nicefrac{1}{2}}Q_2\big(\check{B}(x',t)\big)\,\rmd t\\
-Q_2\left(\int_{\nicefrac{-1}{2}}^{\nicefrac{1}{2}}\check{B}(x',t)\,\rmd t\right)-12Q_2\left(\int_{\nicefrac{-1}{2}}^{\nicefrac{1}{2}}t\check B(x',t)\,\rmd t\right),
\end{multline}
for a.e.\ $x'\in \omega$ and every $G\in \Sym(2)$. 

\medskip

 Before passing to the rigorous derivation of the 2D model, 
we provide
a technical lemma
consisting in two estimates 
for the family $\{W^h\}$ of energy densities and for its
uniform limit $W$ defined in a neighbourhood $\mathcal U$ of $\SO(3)$.
They are elementary consequences
of properties (ii) and (iii)
of Definition \ref{adm_en_dens} (hence we omit their proof). These estimates will be used 
in the proof of the 
$\Gamma\text{-}\liminf$ and 
the $\Gamma\text{-}\limsup$. 

\begin{lemma}\label{prop_wh_w}
 Let $\bar r>0$ be such that $B_{2\bar r}(\mathbb I_3)$ is contained in $\mathcal U$.  Then for every $\varepsilon>0$ there exists $h_{\varepsilon}>0$ and  $C>0$  such that for a.e.\ $x\in \Omega$, every $F\in B_{\bar{r}}(0)$ and every $h\in(0,h_{\varepsilon}]$ it holds that
\begin{equation}\label{Wh-W0}
\bigg|W^h\Big(x,U^h(x)+F\Big)-W\Big(\mathbb{I}_3+F\Big)\bigg|\leq \varepsilon|F|^2,
\end{equation}
\begin{equation}\label{Whbound}
\Big|W^h\Big(x,U^h(x)+F\Big)\Big|\leq  C |F|^2.
\end{equation}
\end{lemma}
 We also introduce two auxiliary functions that will be used in the proof of $\Gamma$-convergence result. Letting $\bar{r}>0$ be as in Lemma \ref{prop_wh_w} above,
we define 
\begin{equation}
\label{rho00}
\rho^0(F):=W( \mathbb{I}_3+F)-\frac{1}2D^2W(\mathbb{I}_3)[F]^2\quad\mbox{ and }\quad \rho(s):=\sup_{|F|\leq s}|\rho^0(F)|
\end{equation}
for every $F\in B_{\bar{r}}(0)$ and every $s>0$ . 
  
As  a direct consequence of
the regularity of $W$,  we have that 
\begin{equation}\label{rho0}
\rho(s)/s^2\to 0\quad\mbox{as } s\to 0.
\end{equation}


\bigskip
In order to state and prove the  following  compactness and 
$\Gamma$-convergence results 
we will use the
standard notation
\[
\nabla'y:=\left(\partial_1y\bigg\vert\,\partial_2y\right)\qquad\mbox{and}\qquad \nabla_hy:=\left(\nabla'y\,\bigg\lvert\,\frac{1}{h}\partial_3y\right).
\]
 Moreover,  given a
$\mathcal{B}\text{-}$admissible 
family $\big\{W^h\big\}_{h>0}$  
of energy densities
in the sense of Definition \ref{adm_en_dens},
for every $h>0$ 
we define the rescaled
\emph{free energy functional} 
$\mathcal{E}^h:\WW^{1,2}(\Omega,\mathbb{R}^3)\to [0,+\infty]$ as
\begin{equation}
\label{elastic_energy}
\mathcal{E}^h(y):=\int_{\Omega}W^h\big(x,\nabla_hy(x)\big)\,\rmd x,
\quad\quad\mbox{for every }\, y\in \WW^{1,2}(\Omega,\mathbb{R}^3).
\end{equation}


\begin{thm}[Compactness]\label{cpt_thm} Let $\big\{y^h\big\}_{h>0}\subseteq \WW^{1,2}\big(\Omega,\mathbb{R}^3\big)$ be a sequence which satisfies 
\begin{equation}\label{cpt_wh}
\limsup_{h\to 0}\frac{1}{h^2} \mathcal E^h(y^h) <+\infty.
\end{equation}
Then $\big\{\nabla_hy^h\big\}_{h>0}$ is precompact in $\LL^2\big(\Omega,\mathbb{R}^{3\times 3}\big)$, that is:
there exists a (not relabeled) subsequence such that  $\nabla_hy^h\to \left(\nabla'y\big|\nu\right)$ in $\LL^2\big(\Omega,\mathbb{R}^{3\times 3}\big)$, where $\nu(x):=\partial_1y(x)\wedge \partial_2y(x)$. Moreover, the limit $\left(\nabla'y\big|\nu\right)$ has the following properties:
\begin{itemize}
\item[$(i)$] $\left(\nabla'y\big|\nu\right)(x)\in \SO(3)$ for a.e.\ $x\in\Omega$,
\item[$(ii)$] $\left(\nabla'y\big|\nu\right)\in \WW^{1,2}\big(\Omega,\mathbb{R}^{3\times 3}\big)$ and
\item[$(iii)$] $\left(\nabla'y\big|\nu\right)$ is independent of $x_3$.
\end{itemize}
 In other words, the limiting deformation $y$ belongs to the class $\Wiso(\omega)$ defined as in \eqref{Aiso}. 
\end{thm}

To prove this compactness result,
we can use the same argument as in the 
proof of the corresponding result
in  \cite{FJM02} where the spontaneous stretch is $\mathbb I_3$ in place of our $U^h=\mathbb I_3+hB$. Note that the same argument holds in the case of spontaneous stretch of the form $\mathbb I_3+h^{\alpha}B$ with $\alpha\geq 1$. 

\begin{proof}[Proof of Theorem \ref{cpt_thm}]
We will show that the sequence $\big\{\nabla_hy^h\big\}_{h>0}\subseteq \LL^2\big(\Omega,\mathbb{R}^{3\times 3}\big)$ satisfies 
\begin{equation}\label{ef_2}
\limsup_{h\to 0}\frac{1}{h^2}\int_{\Omega}{\rm dist}^2\left(\nabla_hy^h(x),\SO(3)\right)\,\rmd x <+\infty.
\end{equation} 
The thesis then directly follows by applying Theorem 4.1 from \cite{FJM02}.
Fix $h>0$ and  $F\in\mathbb{R}^{3\times 3}$. For a.e.\ $x\in \Omega$ there exists $R_{h,F}(x) \in \SO(3)$ such that  
$${\rm dist}\Big(F,\SO(3)\big(\mathbb{I}_3+hB^h(x)\big)\Big)=\big|F-R_{h,F}(x)(\mathbb{I}_3+hB^h\big(x)\big)\big|.$$ We have the following estimate:
\begin{equation}\label{cpt_est}
\begin{aligned}
{\rm dist}^2(F,\SO(3))\leq & \, \big|F-R_{h,F}(x)\big|^2\leq 2\Big|F-R_{h,F}(x)\big(\mathbb{I}_3+hB^h(x)\big)\Big|^2+2\Big|hR_{h,F}(x)B^h(x)\Big|^2\\
 \overset{\text{\eqref{gr_cond_wh}}}{\leq}  & \, \frac{2}{C}W^h(x,F)+6h^2\big|B^h(x)\big|^2
\end{aligned}
\end{equation}
for a.e. $x\in \Omega$.
By \eqref{cpt_wh} and \eqref{cpt_est} we have that \eqref{ef_2} holds true.
\end{proof}

Given a bounded Lipschitz domain 
$S\subset\mathbb{R}^2$, 
the class of the isometries of 
$S$ into $\R^3$ is denoted by  
\begin{equation}\label{Aiso}
{\rm W_{\rm iso}^{2,2}}(S,\R^3)
\,=\,
\Big\{y\in \WW^{2,2}\big(S,\mathbb{R}^3\big):\big|\partial_1y\big|=\big|\partial_2y\big|=1,\,\partial_1y\cdot\partial_2y=0\Big\}.
\end{equation}
For the sake of brevity, 
we equivalently use the symbol 
${\rm W_{\rm iso}^{2,2}}(S)$.
We recall that for a given $y\in \WW^{2,2}\big(\omega,\mathbb{R}^3\big)$ the pull-back of the second fundamental form of $y(\omega)$ at the point $y(x')$ 
is given by 
\begin{equation}
\label{2ff}
{\rm A}_y(x')
\,:=\,(\nabla'y(x'))^{\trsp}\nabla'\nu(x'),\qquad \mbox{where }\; \nu(x'):=\partial_1y(x')\wedge\partial_2y(x')\quad \mbox{ for a.e. } x'\in \omega.
\end{equation}
As we  are going to  see, 
the 2D limiting 
model will depend
on deformations 
$y\in{\rm W_{\rm iso}^{2,2}}
(\omega,\R^3)$
through ${\rm A}_y$. 
 
More precisely, the limiting model
 
will be described by the
energy functional 
$\mathcal E^0:\WW^{1,2}(\Omega,\mathbb{R}^3)
\to[0,+\infty]$
defined as
\begin{equation}
\label{limit_functional}
\mathcal{E}^0(y):=\left\{\begin{aligned}
&\frac{1}{2}\int_{\omega}\overline{Q}_2(x',{\rm A}_y(x'))\,\rmd x',\\
&+\infty,
\end{aligned}
\begin{aligned}
&\quad\mbox{ for } y\in \Wiso(\omega),\\
&\quad\mbox{ otherwise,}
\end{aligned}\right.
\end{equation}
where $\overline Q_2$ is defined 
through \eqref{Bh} and \eqref{Q_2_prima_forma}.

 We also recall that smooth functions are dense in the class of $\WW^{2,2}$-isometric immersions, as stated in the following theorem proved in \cite{Hor111}. 

\begin{thm}\label{densityHornung}
Assume that $S\subseteq\mathbb{R}^2$ is a bounded Lipschitz domain which satisfies \eqref{dom_cond}.
Then $\Wiso(S)\cap C^{\infty}(\overline{S},\mathbb{R}^3)$ is $\WW^{2,2}$-strongly dense in $\Wiso(S)$.
\end{thm}

 This  density result
will be used for the 
construction of the recovery sequence
in the proof of the 
$\Gamma\text{-}\limsup$
convergence result below.

\begin{thm}[\texorpdfstring{$\Gamma$}{gamma}-limit]\label{Gamma_lim}
The following convergence results 
hold true:
\smallskip

{\rm (i)}\ \textsc{$\Gamma\text{-}\liminf$:} 
for every sequence $\{y^h\}_{h>0}$ and every $y$
such that $y^h\rightharpoonup y$ weakly in 
$\WW^{1,2}(\Omega,\mathbb{R}^3)$, it holds 
$$\mathcal{E}^0(y)
\,\leq\,
\liminf_{h\to 0}\frac{1}{h^2}\mathcal{E}^h(y^h),$$

{\rm (ii)}\ 
\textsc{$\Gamma\text{-}\limsup$:} 
under the hypothesis
\begin{equation}\label{comp_Dmin}
{\rm curl}\big({\rm curl}\, D_{{\rm min}}\big)=0\;\mbox{ in }\; \WW^{-2,2}\big(\omega, \Sym(2)\big),
\end{equation} 
with $D_{\rm min}$ defined by 
\eqref{Bh} and \eqref{Dmin},
we have that 
for every $y\in \WW^{1,2}(\Omega, \mathbb{R}^3)$ there exists a sequence $\{y^h\}_{h>0}$ 
such that $y^h\to y$ in $\WW^{1,2}(\Omega, \mathbb{R}^3)$, fulfilling 
$$\mathcal{E}^0(y)
\,=\,
\lim_{h\to 0}\frac{1}{h^2}\mathcal{E}^h(y^h).$$ 
\end{thm}

The convergence results of the 
previous theorem amount to saying that 
the sequence of energy functionals $\frac{1}{h^2}\mathcal{E}^h$ $\Gamma$-converge to $\mathcal{E}^0$, as $h\to 0$, in the strong and weak topology of $\WW^{1,2}(\Omega,\mathbb{R}^3)$. 
The operator ${\rm curl}$ inside the parenthesis 
in condition \eqref{comp_Dmin} acts on
a $2\!\times\!2$ matrix by taking the ${\rm curl}$ of each row, giving as a result a two-dimensional vector.
We postpone the proof of the theorem 
after the following example.

\begin{example}\label{example_1}
{\rm 
Note that when $D_{\rm min}$ is constant,  
condition \eqref{comp_Dmin} is trivially
satisfied. 
In particular, 
recalling definition \eqref{Dmin},
condition \eqref{comp_Dmin}
is trivially satisfied whenever the map
$x\mapsto \check B$ is constant in $x'$.
At the same time, the same condition is
satisfied with $D_{\rm min}\equiv0$
by every map $x\mapsto \check B(x)$
which is nothing but odd in $x_3$.
We also note that it is possible
to realize $D_{\rm min}\equiv C\neq0$
through a map $x\mapsto \check B(x)$ 
which is not constant in $x'$. 
To construct such an example, one can
fix $B_{m}\in\Sym(2)\setminus\{0\}$ and define
\[
\check B(x)
\,:=\,
\sum_{i=1}^{N}\check B_{i}(x_3)\chi_{\omega_{i}}(x'),
\qquad\mbox{for a.e. }x\in \Omega,
\]
where $\{\omega_{i}\}_{i=1}^{N}$ is a partition of $\omega$ and $\{\check B_{i}\}_{i=1}^N
\subseteq
\LL^{\infty}\big((-1/2,1/2),\mathbb{R}^{3\times 3}\big)$ is a family of functions satisfying
$$\int_{\nicefrac{-1}{2}}^{\nicefrac{1}{2}}\check B_{i}(x_3)\,\rmd x_3=B_{m},\qquad\mbox{for every }i=1,\ldots,N,$$
and such that $\check B_i(x_3)\neq \check B_j(x_3)$ 
for every $i\neq j$ and every $x_3$. 
This gives rise to $\check B$ which
is piecewise constant in $x'$ 
(but not constant in the same variable),
and in turn to 
\[
D_{{\rm min}}(x')
\,=\,
\sum_{i=1}^{N}\chi_{\omega_{i}}(x')
\int_{\nicefrac{-1}{2}}^{\nicefrac{1}{2}}\check B_{i}(t)\,\rmd t
\,=\,
B_{m}.
\]

Note also that the above defined map $\check B$ can give rise to a non-constant tensor valued map $x'\mapsto \int_{\nicefrac{-1}{2}}^{\nicefrac{1}{2}}t\check B(x',t)\,\rmd t$, which is interpreted in Section \ref{sec_min} (in each point $x'$) as the target curvature tensor which appears in the 2D limiting model. Indeed, in the case of $N=2$, by choosing $\check B_1(x_3):=(x_3+1)\mathbb {I}_2$ and $\check B_2(x_3):=(x_3^3+1)\mathbb{I}_2$ for all $x_3\in (-1/2,1/2)$, we obtain a simple example of $\check B$ for which $D_{\rm min}$ is constant, while the tensor-valued map $x'\mapsto \int_{\nicefrac{-1}{2}}^{\nicefrac{1}{2}}t\check B(x',t)\,\rmd t$ is piecewise constant.

}
\end{example}

The proof of the $\Gamma\text-\liminf$ is a 
straightforward adaptation
to the case of a family of
energy densities $\{W^h\}$
with wells $\SO(3)\big(\mathbb I_3+hB^h\big)$, 
of the corresponding result in \cite{FJM02}
pertaining the case of a homogeneous $W$
(minimized at $\SO(3)$).
For the construction of 
the recovery sequence in
the proof of the $\Gamma\text-\limsup$ one has 
instead to add an additional term
with respect to the classical construction
(see the third summand on the 
right-hand side of \eqref{recovery_simple}).
Such additional term
gives rise, in the limit as $h\to0$,
to a symmetrized gradient 
(see formula \eqref{formula_Gamma_limsup}),
in a position where the map $D_{\rm min}$ should appear
in order to match the $\Gamma$-limit
(cfr. \eqref{Qbar} and \eqref{limit_functional}).
For this purpose, condition \eqref{comp_Dmin}
guarantees that the map $D_{\rm min}$
is a symmetrized gradient,
thanks to Theorem \ref{Saint_comp}.
Throughout the following proof $\overline C$ is a generic positive constant, varying form line to line and independent of all other quantities.

\begin{proof}[Proof of Theorem \ref{Gamma_lim}] 
(i) \textsc{$\Gamma\text{-}\liminf$:} Let $y\in \WW^{1,2}(\Omega, \mathbb{R}^3)$ and $\{y^h\}$ be such that $y^h\rightharpoonup y$ weakly in $\WW^{1,2}(\Omega, \mathbb{R}^3)$. 
Assume that $\liminf_{h\to 0}\mathcal{E}^h\big(y^h\big)/h^2<+\infty$, otherwise the proof is trivial. 
Then, as shown in \cite{FJM02}
and up to a (not relabeled) subsequence,
there exists a family of piecewise constant maps $R^h:Q_h\to \SO(3)$ such that 
\begin{equation}\label{rh_estimate}
\int_{Q_h\times (-1/2,1/2)}\big|\nabla_hy^h(x)-R^h(x')\big|^2\,\rmd x\leq \overline{C}h^2,
\end{equation} 
and $R^h\to (\nabla' y|\nu)$ in $\LL^2(\Omega,\mathbb{R}^3)$ as $h\to 0$, 
where $Q_h:=\bigcup_{Q_{a,3h}\subseteq \omega}Q_{a,h}$ and  $Q_{a,h}:=a+(-h/2,h/2)^2$  for every $h>0$ and $a\in h\mathbb{Z}^2$.
Moreover, the sequence  
$G^h:\Omega\to \mathbb{R}^{3\times 3}$ defined by
\begin{equation}\label{Gh}
G^h(x',x_3):=\left\{\begin{aligned}
&\frac{R^h(x')^{\trsp}\nabla_hy^h(x',x_3)-\mathbb{I}_3}{h}\\
&0
\end{aligned}
\begin{aligned}
&\,\mbox{ for }x\in Q_h\times (-1/2,1/2),\\
&\,\mbox{ elswhere in } \Omega,
\end{aligned}\right.
\end{equation}
converges weakly in $\LL^2(\Omega,\mathbb{R}^{3\times 3})$, as $h\to 0$, to some $G\in \LL^2(\Omega,\mathbb{R}^{3\times 3})$ 
such that
\[
\check{G}(x)=\check{G}(x',0)+x_3{\rm A}_y(x'),
\qquad\mbox{for a.e.\ }x\in \Omega.
\]
Letting $\chi_h$ be the characteristic function of the set $Q_h\cap \big\{|G^h(x)|\leq 1/\sqrt{h}\big\}$ we also have that $\chi_hG^h\rightharpoonup G$ in $\LL^2(\Omega,\mathbb{R}^{3\times 3})$ as $h\to 0$.
Now,  by denoting $A^h:=G^h-B^h$ and by using also the convergence in \eqref{Bh} we have 
\[
A^h\rightharpoonup G-B\mbox{ in }\LL^2(\Omega,\mathbb{R}^{3\times 3})\quad\mbox{and}\quad \Vert hA^h\Vert_{\LL^{\infty}\big(Q_h\cap \{|G^h(x)|\leq 1/\sqrt{h}\}\big)}\to 0.
\]
By using frame indifference of $W^h$,  \eqref{rho0} and the estimate \eqref{Wh-W0} from Lemma \ref{prop_wh_w}, we have that for a fixed $\varepsilon>0$,  there exists $\bar{h}>0$ such that the following estimates hold for every $h\in(0, \bar{h})$:
\[
\begin{aligned}
\frac{1}{h^2}\int_{\Omega}W^h(x,\nabla_hy^h(x))\,\rmd x
\geq &\frac{1}{h^2}\int_{\Omega}\chi_hW^h(x,R^h(x')^{\trsp}\nabla_hy^h(x))\,\rmd x\\
= &\frac{1}{h^2}\int_{\Omega} \chi_hW^h\Big(x,\big(\mathbb{I}_3+hB^h(x)\big)+hA^h(x)\Big)\,\rmd x\\
\geq &\frac{1}{h^2}\int_{\Omega}\chi_h\frac{1}{2}D^2W(\mathbb{I}_3)\big[hA^h(x)\big]^2-\chi_h\varepsilon|hA^h(x)|^2 +\chi_h\rho^0\big(hA^h(x)\big)\,\rmd x\\
\geq&\int_{\Omega}\chi_h \frac{1}{2}Q_3\big(A^h(x)\big)-\chi_h\varepsilon|A^h(x)|^2-\chi_h\frac{\rho\big(\big|hA^h(x)\big|\big)}{|hA^h(x)|^2}|A^h(x)|^2\,\rmd x,
\end{aligned}
\]
where $\rho^0$ and $\rho$ are defined   in \eqref{rho00} . Since $Q_3$ is lower semicontinuous in the weak topology of $\LL^2(\Omega,\mathbb{R}^{3\times 3})$  and since \eqref{rho0} holds , passing to $\liminf$ as $h\to 0$  in the above inequality we obtain
$$\liminf_{h\to 0}\frac{1}{h^2}\int_{\Omega}W^h(x,\nabla_hy^h(x))\,\rmd x\geq \int_{\Omega}\frac{1}{2}Q_3\big(G(x)-B(x)\big)\,\rmd x-\overline{C}\varepsilon, $$
where $\overline{C}>0$ is such that $||A^h||_{\LL^2(\Omega, \mathbb{R}^{3\times 3})}\leq \overline{C}$.
Finally, by letting $\varepsilon\to 0$ and by using the fact that $Q_3(F)\geq Q_2(\check{F})$ for every $F\in \mathbb{R}^{3\times 3}$  we get that
\[\begin{aligned}
\liminf_{h\to 0}\frac{1}{h^2}\int_{\Omega}W^h(x,\nabla_hy^h(x))\,\rmd x\geq &\frac{1}{2}\int_{\Omega}Q_2\big(\check{G}(x',0)+x_3{\rm A}_y(x')-\check{B}(x',x_3)\big)\,\rmd x\\
\geq & \frac{1}{2}\int_{\omega}\overline{Q}_2\big(x',{\rm A}_y(x')\big)\,\rmd x',
\end{aligned}
\]
which proves $\Gamma\text{-}\liminf$ inequality.

(ii) \textsc{$\Gamma\text{-}\limsup$:} Let us prove $\Gamma\text{-}\limsup$ inequality for a given $y\in \WW^{2,2}_{{\rm iso},0}(\omega):=\Wiso(\omega)\cap C^{\infty}(\overline{\omega},\mathbb{R}^3)$. Once we have proved it, $\Gamma$-$\limsup$ inequality will follow for any $y\in \Wiso(\omega)$ by the density result of Theorem \ref{densityHornung} and the continuity of the limiting functional $\mathcal E^0$ with respect to $\WW^{2,2}$ convergence.
Suppose that $\mathcal{E}^0(y)<+\infty$ (otherwise the proof is trivial). Let $d\in C^{\infty}_{c}(\Omega,\mathbb{R}^3)$ and define $D:\Omega\to \mathbb{R}^3$ by
\[D(x',x_3):=\int_0^{x_3}d(x',t)\,\rmd t,\qquad\mbox{for every }(x',x_3)\in \omega\times (-1/2,1/2)=\Omega.\]
Let $\tilde{g}:=(\tilde{g}_1,\tilde{g}_2)\in C^{\infty}_{c}(\mathbb{R}^2,\mathbb{R}^2)$. We consider the family of functions $y^h$ of the form
\begin{equation}
\label{recovery_simple}
y^h(x):=y(x')+h\big[x_3\nu(x')+\nabla'y(x')\tilde{g}(x')\big]+h^2D(x',x_3),
\end{equation}
for every $x\in \Omega$ and every $h>0$, whose ($h$-rescaled) gradient $\nabla_hy^h$ reads as
\[
\nabla_hy^h(x)=(\nabla'y(x')\big|\nu(x'))+h\big(\nabla' \big[x_3\nu(x')+\nabla' y(x') \tilde{g}(x')\big]\big|d(x)\big)+h^2\big(\nabla'D(x)\big|0\big),
\]
for every $x\in \Omega$ and every $h>0$,
One can easily verify that, in particular, $\{y^h\}_{h>0}\subseteq \WW^{2,\infty}(\Omega, \mathbb{R}^3)$ and that it converges in $\WW^{1,2}$ to $y$, as $h\to 0$.
Denote by $R(x'):=\big(\nabla' y(x')\big|\nu(x')\big)$ for every $x'\in \omega$. 
Set 
\[
C^h(x):=R^{\trsp}(x')\Big(\big(\nabla'\big[x_3\nu(x')+\nabla'y(x') \tilde{g}(x')\big]\big|d(x)\big)+h\big(\nabla' D(x)\big|0\big)\Big)-B^h(x),\quad\mbox{ for a.e.\ } x\in \Omega,
\]
and note that
$C^h$ converges in  $\LL^{\infty}(\Omega,\mathbb{R}^{3\times 3})$ to the function 
\[
\Omega\ni x\mapsto  R^{\trsp}(x')\big(\nabla'\big[x_3\nu(x')+\nabla'y(x') \tilde{g}(x')\big]\big|d(x)\big)-B(x)\in \mathbb{R}^{3\times 3}.
\]
With this notation, we have that $R^{\trsp}(x')\nabla_h y^h(x)=U^h(x)+hC^h(x)$ for a.e.\ $x\in  \Omega$ with $U^h$ given by \eqref{spont_strain}.
By the frame indifference of $W^h(x,\cdot)$, boundedness of $C^h$ and $B^h$ in $\LL^{\infty}$-norm and the estimates \eqref{Wh-W0} and \eqref{Whbound} from Lemma \ref{prop_wh_w}, there exists $\overline C,\bar{h}>0$ such that
\[
\frac{1}{h^2}W^h\Big(x,\nabla_hy^h(x)\Big)=\frac{1}{h^2}W^h\Big(x,R^{\trsp}(x')\nabla_h y^h(x)\Big)=\frac{1}{h^2}W^h\Big(x,U^h(x)+hC^h(x)\Big)\leq \overline{C},
\]
for a.e.\ $x\in \Omega$ and every $0<h\leq \overline{h}$.
Moreover, 
\[\frac{1}{h^2}W^h\Big(x,\nabla_hy^h(x)\Big)\to \frac{1}{2}Q_3\Big(R^{\trsp}(x')\big(\nabla'\big[x_3\nu(x')+\nabla'y(x') \tilde{g}(x')\big]\big|d(x)\big)-B(x)\Big)\]
pointwise almost everywhere in $\Omega$, as $h\to 0$.
Then, by applying dominated convergence theorem we get that
\[
\frac{1}{h^2}\int_{\Omega}W^h\Big(x,\nabla_hy^h(x)\Big)\,\rmd x 
\to 
\frac{1}{2}\int_{\Omega}Q_3\Big(R^{\trsp}(x')\big(\nabla'\big[x_3\nu(x')+\nabla'y(x') \tilde{g}(x')\big]\big|d(x)\big)-B(x)\Big)\,\rmd x
\] 
as $h\to 0$.
 To proceed, for a.e.\ $x\in \Omega$ we denote by $\overline F(x)$ the $2\times 2$ part of the argument of $Q_3$ in the above integral.   
Now let $\ell:\Sym(2)\to \R^3$ be the map that associates to every $F\in \Sym(2)$ the unique element of ${\rm argmin}_{c\in \R^3}\,Q_3\big(\hat F+(c\otimes {\sf f}_3)_{\sym}\big)$. By writing down the first order necessary condition for the minimum problem defining $\ell(F)$, one can easily deduce that the map $\ell$ is linear.
Define
$\overline d:\Omega\to \R^3$ as
\[
\overline d(x):=R(x')\left(\ell\big(\overline F_{\sym}(x)\big)+(2B_{13}\;\; 2B_{23}\;\; B_{33})^{\trsp}(x)-\left(\begin{matrix}\big(\nabla'\big[x_3\nu(x')+\nabla'y(x') \tilde{g}(x')\big]\big)^{\trsp}\nu(x')\\ 0\end{matrix}\right)\right),
\]
for a.e.\ $x\in \Omega$.
Since $\overline F_{\sym}\in \LL^{\infty}(\Omega, \Sym(2))$, $B\in \LL^{\infty}(\Omega,\Sym(3))$ and $y$ and $\tilde g$ are smooth vector fields, it follows that $\overline d$ belongs to $\LL^2(\Omega, \R^3)$.
By choosing $d$ to be equal to $\bar d$, one can readily check that
\begin{equation}\label{ell_barF}
\Big(R^{\trsp}(x')\big(\nabla'\big[x_3\nu(x')+\nabla'y(x') \tilde{g}(x')\big]\big|\overline d(x)\big)-B(x)\Big)_{\sym}=\left(
\begin{array}{c|c}
\overline F_{\sym}(x) & \begin{array}{c}
0\\
0
\end{array} \\
\hline
0\quad 0 & 0
\end{array}
\right)+\Big(\ell\big(\overline F_{\sym}(x)\big)\otimes {\sf f}_3\Big)_{\sym}.
\end{equation}
Observe further that
\[
\begin{aligned}
(\nabla'y)\tilde{g}=&\,\tilde{g}_1\partial_1y+\tilde{g}_2\partial_2y,\\ 
\nabla'\big((\nabla'y)\tilde{g}\big)=&\,\big(\tilde{g}_1\partial_1\partial_1y+\tilde{g}_2\partial_1\partial_2y\big\vert\,\tilde{g}_1\partial_2\partial_1y+\tilde{g}_2\partial_2\partial_2y\big)+\nabla'y\nabla'\tilde{g}.
\end{aligned}
\]
Since $y\in \WW^{2,2}_{{\rm iso},0}(\omega)$, it holds that  
$\partial_{i}y\cdot\partial_{j}y=\delta_{ij}$ and $\partial_{i}\partial_{j}y\cdot\partial_{k}y=0$ for every $i,j,k=1,2$. In turn, we have that
\begin{equation*}\label{nabla_g}
(\nabla'y)^{\trsp}\,\nabla'\big((\nabla'y)\tilde{g}\big)=\nabla'\tilde{g}.
\end{equation*}
Now, by direct computation we obtain
 
\begin{equation}\label{Fsym}
\begin{aligned}
\overline F_{\sym}(x)=
 \,x_3{\rm A}_{y}(x')+\nabla'_{\sym}\tilde g(x')-\check B(x),\qquad\mbox{ for a.e.\ }x\in \Omega.
\end{aligned}
\end{equation}
Finally, by definition of $Q_2$, \eqref{ell_barF} and \eqref{Fsym} it holds that  
  
\[
\begin{aligned}
&\int_{\Omega}Q_3\Big(R^{\trsp}(x')\big(\nabla'\big[x_3\nu(x')+\nabla' y(x') \tilde{g}(x')\big]\big| \overline d(x) \big)-B(x)\Big)\,\rmd x\\
=&\int_{\Omega}Q_2\big(x_3{\rm A}_y(x')+\nabla'_{{\rm sym}}\tilde{g}(x')-\check{B}(x)\big)\,\rmd x.
\end{aligned}
\]
Therefore, the density of $C_{c}^{\infty}(\Omega,\mathbb{R}^3)$ in $\LL^2(\Omega,\mathbb{R}^3)$  and 
a diagonal argument give us that 
\begin{equation}
\label{formula_Gamma_limsup}
\limsup_{h\to 0}\frac{1}{h^2}\int_{\Omega}W^h(x,\nabla_hy^h(x))\,\rmd x=\int_{\omega}\frac{1}{2}\int_{\nicefrac{-1}{2}}^{\nicefrac{1}{2}}Q_2\big( x_3{\rm A}_y(x')+\nabla'_{{\rm sym}}\tilde{g}(x')-\check{B}(x',x_3)\big)\,\rmd x_3\,\rmd x'.
\end{equation}
The compatibility assumption \eqref{comp_Dmin} on $D_{{\min}}$ and Theorem \ref{Saint_comp} guarantee the existence of the map $w\in \WW^{1,2}(\omega,\mathbb{R}^2)$ such that  
$D_{{\min}}(x')=\nabla_{{\rm sym}}w(x')$ for a.e. $x'\in \omega$. 
Thus, by using the density of $C_{c}^{\infty}(\mathbb{R}^2,\mathbb{R}^2)$ (when restricted to $\omega$) in $\WW^{1,2}(\omega, \mathbb{R}^2)$ and a diagonal argument one more time, we prove $\Gamma$-$\limsup$ inequality for a given $y\in \WW^{2,2}_{{\rm iso},0}(\omega)$.
\end{proof}

The following result is used
in the proof of the $\Gamma\text{-}\limsup$.
It can be found in 
 \cite[Theorem 3.2]{Ciarlet2005}  and 
has to do with the so-called  
Saint-Venant compatibility condition in  $\LL^2$ .

\begin{thm}\label{Saint_comp}
Let $S\subseteq \mathbb{R}^2$ be a simply-connected bounded domain with Lipschitz boundary  and let  $A\in \LL^2(S,\Sym(2))$ . Then 
 
\begin{equation}
{\rm curl}\big({\rm curl}\, A\big)=0\;\mbox{ in }\;\WW^{-2,2}(S,\Sym(2))\quad\Longleftrightarrow\quad A=\nabla_{{\rm sym}}w\;\mbox{ for some }\;w\in \WW^{1,2}(S, \mathbb{R}^2).
\end{equation}
 
Moreover $w$ is unique up to rigid displacements.
\end{thm}

\begin{remark} {\rm By standard arguments of $\Gamma$-convergence it can be shown that the above analysis holds also in the case when the appropriate body forces are present. More precisely, the above results can be applied to the sequence of functionals $\{\mathcal F^h\}_{h>0}$ defined by 
$$\mathcal{F}^h(y)=\mathcal{E}^h(y)-\int_{\Omega}f^h(x)\cdot y(x)\,\rmd x,\qquad\mbox{ for every } y\in \WW^{1,2}(\Omega, \mathbb{R}^3),$$ 
where $\{f^h\}_{h\geq 0}\subseteq \LL^2(\Omega,\mathbb{R}^3)$ is the family of body forces such that
$$\frac{f^h}{h^2}\rightharpoonup f^0\quad\mbox{ weakly  in }\LL^2(\Omega,\mathbb{R}^3)\quad\mbox{ and }\quad \int_{\Omega}f^h(x)\,\rmd x=0\,\mbox{ for every } h\geq 0.$$
The sequence $\{\mathcal F^h\}$ $\Gamma-$converges, as $h\to 0$, to 
\[\mathcal{F}^0(y):=\left\{\begin{aligned}
&\mathcal{E}^0(y)-\int_{\omega}f(x')\cdot y(x')\,\rmd x',\\
&+\infty,
\end{aligned}
\begin{aligned}
&\quad\mbox{ for } y\in \Wiso(\omega),\\
&\quad\mbox{ otherwise.}
\end{aligned}\right.\]
where $f(x'):=\int_{\nicefrac{-1}{2}}^{\nicefrac{1}{2}}f^0(x',t)\,\rmd t$ for a.e. $x'\in \omega$. }
\end{remark}


\section[2D energy minimizers]
{2D energy minimizers}
\label{sec_min}
\subsection[\texorpdfstring{$x'$}{x}-dependent target 
curvature tensor \texorpdfstring{$\overline{A}$}{A} and pointwise minimizers]{\texorpdfstring{$x'$}{x}-dependent target curvature tensor \texorpdfstring{$\overline{A}$}{A} and pointwise minimizers}
In this section, we discuss the minimizers of the derived $2D$ model in some special cases. Recall that the $2D$ limiting energy functional $\mathcal{E}^0$ is given by  
\[\mathcal{E}^0(y)=\left\{\begin{aligned}
&\frac{1}{2}\int_{\omega}\overline{Q}_2\big(x',{\rm A}_y(x')\big)\,\rmd x',\\
&+\infty,
\end{aligned}
\begin{aligned}
&\qquad\mbox{ for } y\in \Wiso(\omega),\\
&\qquad\mbox{ otherwise,}
\end{aligned}\right.\]
where $\Wiso(\omega)$ is the set of $\WW^{2,2}$-isometric immersions of $\omega$ into $\mathbb{R}^3$, defined by \eqref{Aiso}. From formula \eqref{limit_Q}, we have that 
\begin{equation}
\label{energia_limite}
\mathcal{E}^0(y)
\,=\, 
\frac{1}{24}\int_{\omega}Q_2\bigg({\rm A}_y(x')-12\int_{\nicefrac{-1}{2}}^{\nicefrac{1}{2}}t\check{B}(x',t)\,\rmd t\bigg)\,\rmd x'+{\rm ad.t.} 
\end{equation}
for every $y\in \Wiso(\omega)$, where ${\rm ad.t.}$ stays for ``additional terms" (not depending on $y$) 
 and is given by
\begin{equation}\label{adt_limit_energy}
{\rm ad.t.}\,:=\,\frac{1}{2}\int_{\omega}\int_{\nicefrac{-1}{2}}^{\nicefrac{1}{2}}Q_2\big(\check{B}(x',t)\big)\,\rmd t
-Q_2\left(\int_{\nicefrac{-1}{2}}^{\nicefrac{1}{2}}\check{B}(x',t)\,\rmd t\right)-12Q_2\left(\int_{\nicefrac{-1}{2}}^{\nicefrac{1}{2}}t\check B(x',t)\,\rmd t\right)\rmd x'.
\end{equation} 
 Recall that ${\rm A}_{y}$ is 
the pull-back of the second fundamental form 
associated with $y(\omega)$
(see \eqref{2ff}), 
hence it gives information on
the curvature \emph{realized} by the deformation $y$. 
On the other hand, 
when reading the expression for $\mathcal{E}^0$,
it is natural to define the \emph{target} curvature tensor 
\begin{equation}
\label{def:target_curvature}
\overline{A}(x'):=12\int_{\nicefrac{-1}{2}}^{\nicefrac{1}{2}}t\check{B}(x',t)\,\rmd t,
\qquad\mbox{ for a.e. }x'\in \omega,
\end{equation}
which encodes the spontaneous curvature
of the system.    
While, for a.e.\ $x'$, the tensor
$\overline A(x')$
(which depends on 
$\check B$ and in turn on the 
family of spontaneous strains $\{B^h\}$, 
see formula \eqref{Bh})
is a given 
$2{\times}2$ symmetric matrix
with possibly nonzero determinat,
it is a well known result of differential geometry 
that every smooth $y\in \Wiso(\omega)$ satisfies 
$\det{\rm A}_y=0$ in $\omega$.
From \cite[Lemma 2.5]{Pak04}, one can deduce that the same 
property holds for any arbitrary $y\in\Wiso(\omega)$,
a.e.\ in $\omega$.
Our aim is to determine explicitly some classes of minimizers. 
More precisely, 
introducing the notation
\begin{equation}
\label{sec_fund_form_cyl}
\mathcal{F}
\,:=\,
\big\{F\in \Sym(2):\det F=0\big\},
\end{equation}
and having in mind the inequality
\begin{equation*}
\min_{\Wiso(\omega)}\mathcal{E}^0
\,\geq\,
\frac{1}{24}\int_{\omega}\min_{F\in \mathcal{F}}Q_2\big(F-\overline{A}(x')\big)\,\rmd x'+{\rm ad.t.},
\end{equation*}
we will focus our attention on 
\emph{pointwise minimizers} of $\mathcal{E}^0$.
Namely, on those $y\in \Wiso(\omega)$ such that
\begin{equation}
\mathcal{E}^0(y)
\,=\, 
\frac{1}{24}\int_{\omega}\min_{F\in \mathcal{F}}Q_2\big(F-\overline{A}(x')\big)\,\rmd x'+{\rm ad.t.}
\,\,=
\min_{\Wiso(\omega)}\mathcal{E}^0.
\end{equation}
To go on, let us consider the set 
\begin{equation}\label{setN}
\mathcal{N}(x'):=\underset{F\in\mathcal{F}}{\rm argmin}\,Q_2\big(F-\overline{A}(x')\big),
\end{equation}
for a.e.\ $x'\in \omega$. 
Note that $\mathcal{N}(x')\neq\emptyset$
for a.e.\ $x'\in \omega$, because
$Q_2$ is a positive definite quadratic form 
(when restricted to $\Sym(2)$) 
and $\mathcal{F}$ is a closed subset of $\Sym(2)$. 
To accomplish our program, we would like to
have some explicit representation of the 
elements of $\mathcal{N}(x')$, 
for a.e.\ $x'\in \omega$, 
also in view of the application
which motivates our analysis (see Section \ref{gel_sheets}.).
Therefore, we restrict our attention to case of 
$W$ \emph{isotropic}, i.e. such that 
$$W(RFP)=W(F),\quad\mbox{ for every } F\in \mathbb{R}^{3\times 3} \mbox{ and every } R,P\in \SO(3).$$
This implies the existence of constants
$\lambda\in \mathbb{R}$
and $\mu>0$,
called Lam\' e moduli, such that 
$$Q_3(F)
\,:=\,
D^2W(\mathbb{I}_3)[F,F]
\,=\,2\mu|F_{{\rm sym}}|^2+\lambda\,{\rm tr}^2F,$$
for every $F\in \mathbb{R}^{3\times 3}$
(see \cite{Gur}).
In turn, from this expression one can
easily show that
\begin{equation}\label{Qgel}
Q_2(F)
\,:=\,
\min_{d\in\mathbb{R}^3}Q_3\big(\hat{F}+d\otimes \mathsf{f}_3\big)
\,=\,
2\mu\left(|F_{{\rm sym}}|^2+\beta\,{\rm tr}^2F\right),
\qquad\mbox{for every }F\in \mathbb{R}^{2\times 2},
\end{equation} 
where $\beta$ has the expression
\begin{equation}
\label{beta}
\beta\,=\,
\frac{\lambda}{2\mu+\lambda}.
\end{equation}
 
Since $Q_3$ is positive definite by its very definition, then we have that $\mu>0$ and $2\mu+3\lambda>0$. In turn, it holds that $\beta>-1/2$ and hence that $Q_2$ is positive definite. 
This fact guarantees in particular
that the quantities appearing in the statement
of Lemma \ref{minimizers_x'} below are well defined.

Note that in the case when $\overline{A}$ is constant in $\omega$, 
pointwise minimizers of $\mathcal{E}^0$ always exist.
More precisely, 
as noticed in \cite{Sch071} and \cite{Sch072}
(see Lemma \ref{min_A_const} below), 
any minimizer $y$ of $\mathcal E^0$
with $\overline{A}$ constant is characterized by 
the property
${\rm A}_y(x')\equiv{\rm const.}\in \mathcal{N}$ 
for a.e. $x'\in \omega$, where
\[
\mathcal N
\,:=\,
\underset{F\in \mathcal{F}}{\rm argmin}\,Q_2\big(F-\overline{A}\big).
\]
Clearly, in the case of nonconstant $\overline{A}$, 
this is not always true.
Now, 
while the analysis of the \emph{minimizers} of 
$\mathcal E^0$, 
with an arbitrary nonconstant $\overline{A}$, 
is behind the scope of the present paper,
it is natural in our context 
to try to understand 
under which conditions
the existence of 
\emph{pointwise minimizers} of $\mathcal{E}^0$
is guaranteed. 
In Subsection \ref{subsec_min} we answer this question in the case when $\overline{A}$ is piecewise constant.  
To do this, 
we need a structure result for
the set $\mathcal N$ in the case of constant
$\overline A$. This is the content of the following
lemma.

\begin{lemma}
\label{minimizers_x'} 
Let $a$ and $b$ be two real numbers and
let $\beta$ be given by \eqref{beta}.
The following implications hold:
\noindent
\begin{itemize}
\item [$(i)$] If 
$\,\overline{A}=\left(\begin{matrix}
a & 0 \\ 0 & a
\end{matrix}\right)\quad\,\mbox{ then }\quad \mathcal{N}=\left\{\rho^{\trsp}\left(\begin{matrix}
\mathfrak{r} & 0 \\ 0 & 0
\end{matrix}\right)\rho\,:\,\rho\in \SO(2)\right\}\,
\mbox{with}\ \,\mathfrak{r}=a\frac{1+2\beta}{1+\beta}.$
\item [$(ii)$] If  $\, \overline{A}=\left(\begin{matrix}
a & 0 \\ 0 & -a
\end{matrix} \right)\quad\,\mbox{ then }\quad \mathcal{N}=\left\{\left(\begin{matrix}
\mathfrak{r} & 0 \\ 0 & 0
\end{matrix}\right), \left(\begin{matrix}
0 & 0 \\ 0 & -\mathfrak{r}
\end{matrix}\right)\right\}\,
\mbox{with}\ \,
\mathfrak{r}=\frac{a}{1+\beta}.$
\item [$(iii)$] If $\, \overline{A}=\left(\begin{matrix}
a & 0 \\ 0 & b
\end{matrix} \right),\,|a|>|b|\quad\mbox{ then }\quad\mathcal{N}=\left\{\left(\begin{matrix}
\mathfrak{r} & 0 \\ 0 & 0
\end{matrix}\right)\right\}\,
\mbox{with}\ \,
\mathfrak{r}=a+\frac{b\beta}{1+\beta}.$
\item [$(iv)$] If $\, \overline{A}=\left(\begin{matrix}
a & 0 \\ 0 & b
\end{matrix} \right),\, |b|>|a|\quad\mbox{ then }\quad\mathcal{N}=\left\{\left(\begin{matrix}
0 & 0 \\ 0 & \mathfrak{r}
\end{matrix}\right)\right\}\,
\mbox{with}\ \,
\mathfrak{r}=b+\frac{a\beta}{1+\beta}.$
\end{itemize}
\end{lemma}
Before giving the proof of the above statement, 
let us make a couple of comments.
First, note that the lemma,
though restricted to the case of 
$\overline A$ diagonal,  
covers all the interesting cases, 
from the simple observation that,
with abuse of notation,
$\mathcal N_{\overline A}
=
\bar\rho\,\mathcal N_{\overline D}\,\bar\rho^{\trsp}$,
where $\bar\rho\in{\rm Orth}(2)$
is such that 
$\bar\rho^{\trsp}\,
\overline A\,\bar\rho$
coincides with the diagonal matrix 
$\overline D$.
Moreover, interpreting the elements
of $\mathcal N$ as second fundamental
forms of \emph{cylinders} 
(see the discussion below),
the parameter $\mathfrak r$,
when different from zero, 
corresponds
to the nonzero principal curvature.
In this case, 
observe also that, 
with abuse of notation,
the set $\mathcal N_{(ii)}$ 
is never a subset of $\mathcal N_{(i)}$
and that, as for the (two)
elements of $\mathcal N_{(ii)}$,
the elements of $\mathcal N_{(i)}$
are pairwise linearly independent.
This can be easily read off from
the simple fact that 
$$
\mathcal N_{(i)}
\,=\,
\mathfrak r\,\big\{n\otimes n:\, n\in \mathbb{R}^2\ \mbox{with}\ |n|=1\big\}.
$$
Finally, 
the set of the directions
corresponding to $\pm\mathfrak r$
in the cases 
$(i)$, $(ii)$, $(iii)$, and $(iv)$
is given by 
$
\{\rho\,\mathsf e_1:
\rho\!\in\!\SO(2)\},
\{\mathsf e_1,\mathsf e_2\},
\{\mathsf e_1\},
\{\mathsf e_2\},
$
respectively.
This fact can be interpreted saying that,
in order to reduce the energy, 
while in case $(i)$ rolling up along all
the possible directions is equally favorable, 
in the remaining cases the system 
rolls up along the direction 
corresponding to the greater
(in modulus) eigenvalue of
the target curvature tensor 
$\overline{A}$.
\begin{proof}
[Proof of Lemma \ref{minimizers_x'}] 
 
Let $a, b\in \mathbb{R}$ and let $\overline{A}={\rm diag}(a,b)$. By representing any $F\in \Sym(2)$ by $\left(\begin{smallmatrix}\xi & \zeta\\
\zeta & \upsilon\end{smallmatrix}\right)$, $\zeta,\xi,\upsilon\in \R$ and recalling that $Q_2$ is of the form \eqref{Qgel}, the minimization problem 
to be solved is:
\[\begin{aligned}
\min_{F\in\mathcal{F}}\left\{\big|F- \overline{A}\big|^2+\beta\,{\rm tr}^2\big(F- \overline{A}\big)\right\}
=&\min_{\substack{(\xi,\upsilon)\in \mathbb{R}^2,\,\zeta\in \mathbb{R} \\ \xi\upsilon=\zeta^2}}\left\{\bigg|\left(\begin{matrix}
\xi-a & \zeta \\ \zeta & \upsilon-b
\end{matrix} \right)\bigg|^2+\beta\, {\rm tr}^2\left(\begin{matrix}
\xi-a & \zeta \\ \zeta & \upsilon-b
\end{matrix} \right)\right\}
\end{aligned}
\] 
Denote $P:=\big\{(\xi, \upsilon)\in \mathbb{R}^2\big|\, \xi\upsilon\geq 0\big\}$ and define for every $(\xi,\upsilon)\in P$ the function
\[
f(\xi,\upsilon):=(1+\beta)(\xi+\upsilon)^2-2\big(a (1+\beta)+b\beta\big)\xi-2\big( b(1+\beta)+a\beta\big)\upsilon+a^2+b^2+\beta(a+b)^2,
\]
so that the minimization problem becomes  
$\min_{(\xi,\upsilon)\in P}f(\xi,\upsilon).$ 
In the case
when $a\neq b$, $f$ attains its minimum on $\partial P=\big\{(\xi,\upsilon)\in \mathbb{R}^2\,|\, \xi\upsilon=0\big\}$. With this said, $(ii)$, $(iii)$ and $(iv)$ easily follow by straightforward computations.
To prove $(i)$, we first note that the
set of stationary points of $f$ in ${\rm int}(P)$ 
is given by
\[
\left\{\Big(\eta_{\zeta}^{\pm},\eta_{\zeta}^{\mp}\Big)\in \mathbb{R}^2:\,\zeta\in \left[-\frac{|\mathfrak{r}|}{2},\frac{|\mathfrak{r}|}{2}\right]\setminus\{0\}\right\}, 
\]
where
\[\mathfrak{r}=\frac{a(1+2\beta)}{(1+\beta)}
\quad\mbox{and}\quad\eta^{\pm}_{\zeta}:=\frac{\mathfrak{r}}{2}\pm\frac{\sqrt{\mathfrak{r}^2-4\zeta^2}}{2},\qquad \mbox{ for every }\zeta\in \left[-\frac{|\mathfrak{r}|}{2},\frac{|\mathfrak{r}|}{2}\right]\setminus\{0\}.\]
Moreover, the value of
$f$ at these stationary points coincides with the value of $f$ at the boundary of $P$. 
In turn, 
\[
\mathcal N\,=\,\left\{
\left(
\begin{matrix}
\eta_{\zeta}^{\pm} & \zeta\\
\zeta & \eta_{\zeta}^{\mp}
\end{matrix}
\right):\, |\zeta|\leq \frac{|\mathfrak{r}|}{2}\right\}=\left\{\rho^{\trsp}\left(\begin{matrix}
\mathfrak{r} & 0 \\ 0 & 0
\end{matrix}\right)\rho\,:\,\rho\in \SO(2)\right\},\]
concluding the result of point $(i)$, and thus proving the lemma.
 
\end{proof}

To conclude the section, we give
some definitions which will be useful later on.
They regard the 
sub-class of $\Wiso(\omega)$
consisting of \emph{cylinders}.
Given $r\in (0,+\infty]$, we define the map $C_{r}:\mathbb{R}^2\to \mathbb{R}^3$ as
\[
C_{r}(x')\,:=\,\left\{\begin{aligned}&\Big(r \big(\cos(x_1/r)-1\big),r \sin(x_1/r), x_2\Big)^{\trsp},\ && r\in (0,+\infty),\\
&\big(0,x_1,x_2\big)^{\trsp},\ && r=+\infty,
\end{aligned}\right.
\]
for every $x'=(x_1,x_2)\in \mathbb{R}^2$.
Then we define the family of maps 
\begin{equation}\label{family_of_cyliders}
{\rm Cyl}\,:=\,\big\{T_v\circ R\circ C_{r}\circ \rho:\mathbb{R}^2\to \mathbb{R}^3\,\big|\, r\in (0,+\infty],\, T_v\in {\rm Trs}(3),\, R\in \SO(3)\mbox{ and }\rho\in {\rm Orth}(2)\big\}
\end{equation}
and we call its elements \emph{cylinders}. Note that the above defined family of cylinders includes also \emph{planes} - the elements of ${\rm Cyl}$ with $r=+\infty$. 
\begin{remark}
{\rm
Observe that any cylinder $y=T_v\circ R\circ C_r\circ \rho$ maps lines parallel to $\rho^{\trsp}\mathsf e_2$ to the lines of zero curvature - rulings. 
More in general, direct computations give
\begin{equation}\label{gradient_cylinder}
\nabla y(x')\,=\, R\,\nabla C_{r}\big(\rho(x')\big)\,\rho=R\,\left(\begin{array}{cc}
-\sin\left(\frac{ x'\cdot \rho^{\trsp}\mathsf{e}_1}{r}\right) & 0 \\ \cos\left(\frac{x'\cdot \rho^{\trsp}\mathsf{e}_1}{r}\right) & 0 \\ 0 & 1
\end{array}\right)\,\rho,\qquad\mbox{ for all } x'\in \R^2,
\end{equation}
so that 
\[\nabla y(\lambda\rho^{\trsp}\mathsf e_2)=R \left(\begin{array}{cc}
0 & 0 \\ 1 & 0 \\ 0 & 1
\end{array}\right)\rho,
\qquad\quad\mbox{for every }\lambda\in\R.
\]
} 
\end{remark}

By direct computations 
one can see that a map 
$y=T\circ R\,\circ\,C_r\circ \rho\in {\rm Cyl}$
is an isometry
whose second fundamental form
is given by
\begin{equation}
\label{cyl}
{\rm A}_y(x')\,= \,({\rm det}\rho)\,\rho^{\trsp}\left(\begin{matrix}
\frac{1}{r} & 0\\
0 & 0
\end{matrix}\right)\rho,\quad\quad\mbox{ for every }x'\in \mathbb{R}^2.
\end{equation}
Now, let us go back to the set
$\mathcal F$ defined in 
\eqref{sec_fund_form_cyl}.
From \eqref{cyl} and from the simple observation that
$\mathcal{F}$ can be equivalently represented as
$$\mathcal{F}
\,=\,
\mathbb{R}\,\big\{n\otimes n:\, n\in \mathbb{R}^2\ \mbox{with}\ |n|=1\big\}
\,=\,
\mathbb R\,
\left\{\rho^{\trsp}\left(\begin{matrix}
1 & 0 \\ 0 & 0
\end{matrix}\right)\rho\,:\,\rho\in \SO(2)\right\},
$$
one can prove that
the set $\mathcal{F}$
coincides with the set of 
(constant) second fundamental forms 
of cylinders.
This fact can in turn be used to show,
in the case where the target curvature tensor
$\overline A$ is constant,
that
  
\begin{equation}\label{set_minimizers}
y\in\Wiso(\omega)\text{ is a minimizer of } \mathcal E^0\text{ if and only if }
y\text{ is a pointwise minimizer.}
\end{equation}
 
This is the first step
of the proof of Lemma \ref{min_A_const}
below. The second part of the proof
consists then in showing that
\begin{equation}\label{sff_consant}
{\rm A}_y(x')\in\mathcal N
\quad\mbox{for a.e.}\ x'\in\omega
\quad
\Longrightarrow
\quad
{\rm A}_y\equiv{\rm const.}.
\end{equation}
This property is at the core of
our investigations in the following subsection
and can be proved using
some fine properties
of isometric immersions
  
(\cite{Hor111}, \cite{Hor112} and \cite{Pak04}). The proof of the following lemma can be found in \cite[Proposition 4.2]{Sch072}. 

\begin{lemma}\label{min_A_const}
Let $\overline A$ be constant 
(cfr. \eqref{energia_limite}--\eqref{def:target_curvature})
and let
$y\in \Wiso(\omega)$ be a minimizer of $\mathcal{E}^0$. Then $y=v|_{\omega}$ for some $v\in {\rm Cyl}$. 
In particular, $y$ has constant second fundamental form.
\end{lemma} 
\subsection[The case of piecewise constant \texorpdfstring{$\overline{A}$}{Abar}.]{The case of piecewise constant \texorpdfstring{$\overline{A}$}{Abar}.}\label{subsec_min}
In this subsection, 
we consider the case where
the target curvature is  
a piecewise constant tensor valued map $x'\mapsto\overline{A}(x')$.
More precisely, given $n\in \N$, $n\geq 2$, we say that the map $\overline{A}\in \LL^{\infty}\big(\omega,\mathbb{R}^{2\times 2}\big)$ is \emph{piecewise constant} if it is of the form  
\begin{equation}\label{Abar}
\overline{A}\,=\,\sum_{k=1}^n\overline{A}_k\,\chi_{\omega_k}\quad\mbox{ a.e.\ in }  \omega,\qquad\mbox{ with }\, \overline{A}_k=\left(\begin{array}{cc}
a_k & 0 \\ 0 & b_k
\end{array}\right),\quad a_k, b_k\in \mathbb{R},
\end{equation}
where $\{\omega_k\}_{k=1}^n$ is a partition of $\omega$ made of Lipschitz subdomains $\omega_k$  according to Definition \ref{subdivision} .
Clearly, it is convenient distinguishing 
between two different neighboring subdomain only when the corresponding 
spontaneous curvature are different from each other. Namely, we suppose
that $\overline A_k\neq\overline A_j$
for every $k\neq j$ such that 
$\partial\omega_j\cap
\partial\omega_k\neq\emptyset$. 
With such target curvature,
our 2D energy functional takes the form
$$\mathcal{E}^0(y)\,=\,\frac{1}{24}\sum_{k=1}^n\int_{\omega_k}Q_2\big({\rm A}_y(x')-\overline{A}_k\big)\,\rmd x'+{\rm ad.t.},\qquad\mbox{ for every } y\in \Wiso(\omega).$$
We want to determine the conditions 
the map
$x'\mapsto\overline{A}(x')$
has to satisfy in order to guarantee the existence of pointwise minimizers of 
$\mathcal{E}^0$, i.e. to guarantee that
there exists 
$y\in \Wiso(\omega)$
such that 
${\rm A}_y(x')\in 
\mathcal{N}(x')$
for a.e. $x'\in \omega$,
where 
$\mathcal N(x')$
is 
defined 
by \eqref{setN}.
In view of \eqref{Abar},
we equivalently 
look for the necessary and sufficient conditions 
such that 
\begin{equation}\label{y_pointwise} 
\mbox{there exists }
y\in \Wiso(\omega)
\mbox{ such that }
{\rm A}_y(x')\in \mathcal N_k\mbox{ for a.e.\ }x'\in \omega_k,
\mbox{ for all }k=1,\ldots,n,
\end{equation}
where
\begin{equation}\label{Nk}
\mathcal{N}_k\,:=\,{\rm argmin}_{F\in \mathcal{F}}\,Q_2\big(F-\overline{A}_k\big),\qquad \mbox{ for every }k=1,\ldots,n.
\end{equation}
Note from \eqref{sff_consant}
that a deformation satisfying 
\eqref{y_pointwise}
is, roughly speaking,
a ``patchwork" of cylinders.
Therefore,
conditions on 
$\overline A$ guaranteeing
\eqref{y_pointwise}
translates into
conditions
under which cylinders can be patched together resulting into an isometry. 
 This is the content of the main result of the present section, namely of Theorem \ref{general_cond} below. In order to state and prove it, we need a definition and a preliminary lemma. 

\begin{defin}[Lipschitz \texorpdfstring{$n\text{-}$}{n}subdivision]\label{subdivision} 
Fix $n\!\in\!\mathbb{N}$, $n\geq 2$. A family $\{\omega_k\}_{k=1}^n$ of open, bounded and connected subsets of $\mathbb{R}^2$ is said to be a \emph{Lipschitz $n\text{-}$subdivision} of $\omega$ provided it can be obtained \emph{via} 
the following procedure: 
\begin{itemize}
\item Call $\omega_1':=\omega$.
\item 
Suppose that for every  $k=1,\ldots,n-1$ there exists 
a continuous injective curve $\gamma_k:[0,1]\to {\rm cl}(\omega_k')$ such that $\partial\omega_k'\cap [\gamma_k]=\big\{\gamma_k(0),\gamma_k(1)\big\}$ (note that $\gamma_k(0)\neq\gamma_k(1)$) and the two connected components of $\omega_k'\setminus [\gamma_k]$ are Lipschitz. 
Then call 
$\omega_{k+1}'$ one of such connected components.
\item Once the domains $\omega_1',\ldots,\omega_n'$ are defined, let $\omega_k:=\omega_k'\setminus {\rm cl}(\omega_{k+1}')$ for every $k=1,\ldots,n-1$ and let $\omega_n:=\omega_n'$. 
\end{itemize}
In particular, 
the subdomains 
$\omega_1,\ldots,\omega_n$ 
of $\omega$ are Lipschitz domains
such that
$$\omega=\bigcup\limits_{k=1}^n\omega_k\,\cup \,\bigcup\limits_{k=1}^{n-1}\gamma_k\big((0,1)\big).$$
\end{defin}
\begin{remark}
{\rm
Since each $\omega_k$ is a Lipschitz domain, one has that its boundary $\partial\omega_k$ has null $\mathcal{L}^2$-measure. In particular, we deduce that $\mathcal{L}^2\big(\omega\setminus\bigcup_{k=1}^n\omega_k\big)=0$. 
}
\end{remark} 
 The following Lemma \ref{iso_piecewise}
will be the main ingredient for the proof of Theorem \ref{general_cond}. 
It gives a ``recipe" on 
how two 
cylinders can be patched together.
We refer to Remark \ref{trsrot} below for the notation and the properties of roto-translations used in this section. 
We point out that the proof of Theorem \ref{general_cond} will be achieved via an induction argument, which relies upon Lemma \ref{iso_piecewise}
and the definition of Lipschitz subdivision of $\omega$.

We remark that, more in general, the very fact that $y$ 
is a $\WW^{2,2}$-isometry is sufficient to deduce that
$\omega$ consists, up to a null set, 
of finitely many subdomains (touching each other on a finite union of line segments) on which $y$ is either a plane, or a cylinder, 
or a cone or ``tangent developable".
This description can be obtained
as a consequence of some fine properties of the 
class $\Wiso(\omega)$ - see \cite{Pak04} for $\omega$ convex and \cite{Hor111}, \cite{Hor112} for a more general 
$\omega$. 
%
%
\begin{lemma}
\label{iso_piecewise}
 Let $\gamma:[0,1]\to {\rm cl}(\omega)$ be a continuous injective curve such that $[\gamma]\cap \partial\omega=\{\gamma(0), \gamma(1)\}$ and such that two connected components $\omega_1$ and $\omega_2$ of $\omega\setminus [\gamma]$ are Lipschitz. 
Let $y_1,y_2\in {\rm Cyl}$, say 
$y_1=T_{v_1}\circ R_1\circ C_{r_1}\circ \rho_1$ and 
$y_2=T_{v_2}\circ R_2\circ C_{r_2}\circ \rho_2$,
with $r_1,r_2\in (0,+\infty)$ 
such that ${\rm det}\rho_1=-{\rm det}\rho_2$ whenever $r_1=r_2$.
The map defined as
\[
y
\,:=\,
y_1\chi_{\omega_1}+y_2\chi_{\omega_2},
\ \qquad\mbox{ a.e.\ in }\omega,
\]  
belongs to $\Wiso(\omega)$ if and only if 
the following conditions hold:
\begin{itemize} 

\item [$(i)$] $[\gamma]$ is a line segment
spanned by some $\mathsf e\in\R^2\setminus\{0\}$
;

\item [$(ii)$] 
$\rho^{\trsp}_1\mathsf{e}_2$ and $\rho^{\trsp}_2\mathsf{e}_2$ 
are parallel to $\mathsf e$. 
This in particular implies
that $\rho_1\rho_2^{\trsp}=
{\rm diag}(\sigma_1,\sigma_2)$,
for some $\sigma_1,\sigma_2\in\{\pm1\}$;
 
\item [$(iii)$]
Setting $w_k:=\rho_k\big(\gamma(0)-(0,0)\big)$
and $\theta_k:=(w_k\cdot\mathsf e_1)/r_k$,
for $k=1,2$, we have
\begin{equation}
\label{caso_generale}
\big(R_1\hat R_{\theta_1}\big)^{\trsp}
\big(R_2\hat R_{\theta_2}\big)
=\,
{\rm diag}
\big(\sigma_1\,\sigma_2,\sigma_1,\sigma_2\big)
\quad\mbox{and}\quad
v_1+R_1C_{r_1}(w_1)
=\,
v_2+R_2C_{r_2}(w_2).
\end{equation}
\end{itemize}
\end{lemma}

\begin{proof}

\textbf{(Necessity)}
Here, we show that if the deformation
$y:=
y_1\chi_{\omega_1}+y_2\chi_{\omega_2}$ is in $\Wiso(\omega)$, then it complies with conditions $(i)$, $(ii)$ and $(iii)$.
First of all, 
we recall from 
\cite[Proposition 5]{MP05} that
the very condition $\Wiso(\omega)$
implies 
$y\in C^1(\omega,
\mathbb{R}^3)$.
At the same time, 
from the specific expression of $y$
we have that
$\nabla y=\nabla y_k$ in $\omega_k$
for $k=1,2$, where
\begin{equation}
\label{expr_grad}
\nabla y_k
\,=\,
R_k
\left(\begin{array}{cc}
-\sin\left(\frac{ x'\cdot \rho_k^{\trsp}\mathsf{e}_1}{r_k}\right) & 0 \\ \cos\left(\frac{ x'\cdot \rho_k^{\trsp}\mathsf{e}_1}{r_k}\right) & 0 \\ 0 & 1
\end{array}\right)
\rho_k.
\end{equation}
This expression says in particular that
$\nabla y$ is bounded and in turn
that $y\in C^1(\overline{\omega},
\mathbb{R}^3)$. 
Let us first prove the necessity of the conditions $(i)$, $(ii)$ and $(iii)$ in the case when $\gamma(0)=(0,0)$. The continuity of $y$ and $\nabla y$ at the point $(0,0)$ gives,
respectively,
that $v_1=v_2$ (obtained by imposing $y_1(0,0)=y_2(0,0)$),
and 
\begin{equation}\label{Rrho}
\left(\begin{array}{cc}
0 & 0\\ 1 & 0\\ 0 & 1
\end{array}\right)\rho_1\rho_2^{\trsp}=R_1^{\trsp}R_2 \left(\begin{array}{cc}
0 & 0\\ 1 & 0\\ 0 & 1
\end{array}\right)\;\Leftrightarrow\;R_1^{\trsp}R_2=\left(\begin{array}{cccc}
{\rm det}(\rho_1\rho_2^{\trsp}) &\temp & 0& 0\\ \cline{1-4}
\begin{array}{c}
0\\
0
\end{array} &\temp & \quad \rho_1\rho_2^{\trsp}\\
\end{array}\right)
\end{equation}
(obtained from 
$\nabla y_1(0,0)=\nabla y_2(0,0)$
and from expression \eqref{expr_grad}),  which proves $(iii)$ .
The continuity of $\nabla y$ gives also that
$\nabla  y_1(\gamma(t))=\nabla y_2(\gamma(t))$
for each $t\in[0,1]$, that is 
\[ 
\left(\begin{array}{cc}
-\sin\left(\frac{ \gamma(t)\cdot \rho_1^{\trsp}\mathsf{e}_1}{r_1}\right) & 0 \\ \cos\left(\frac{ \gamma(t)\cdot \rho_1^{\trsp}\mathsf{e}_1}{r_1}\right) & 0 \\ 0 & 1
\end{array}\right)
\rho_1\rho_2^{\trsp}=R_1^{\trsp}R_2
\left(\begin{array}{cc}
-\sin\left(\frac{ \gamma(t)\cdot \rho_2^{\trsp}\mathsf{e}_1}{r_2}\right) & 0 \\ \cos\left(\frac{ \gamma(t)\cdot \rho_2^{\trsp}\mathsf{e}_1}{r_2}\right) & 0 \\ 0 & 1
\end{array}\right).
\]
In turn, using the second condition in \eqref{Rrho}
and the notation
$\rho_1\rho_2^{\trsp}=\left(\begin{matrix}
m_1 & m_2\\
m_3 & m_4
\end{matrix}\right)$,
we have
\renewcommand\arraystretch{1.5}
\begin{equation}\label{nabla_y}
\left(\begin{array}{cc}
-m_1\sin\left(\frac{ \gamma(t)\cdot \rho_1^{\trsp} \mathsf{e}_1}{r_1}\right) & -m_2\sin\left(\frac{\gamma(t)\cdot \rho_1^{\trsp}\mathsf{e}_1}{r_1}\right) \\ m_1\cos\left(\frac{\gamma(t)\cdot \rho_1^{\trsp}\mathsf{e}_1}{r_1}\right) & m_2\cos\left(\frac{\gamma(t)\cdot \rho_1^{\trsp}\mathsf{e}_1}{r_1}\right) \\ m_3 & m_4
\end{array}\right)
=
\left(\begin{array}{cc}
-{\rm det}(\rho_1\rho_2^{\trsp})\sin\left(\frac{\gamma(t)\cdot \rho_2^{\trsp}\mathsf{e}_1}{r_2}\right) & 0 \\ m_1\cos\left(\frac{\gamma(t)\cdot \rho_2^{\trsp}\mathsf{e}_1}{r_2}\right) & m_2\\ m_3 \cos\left(\frac{\gamma(t)\cdot \rho_2^{\trsp}\mathsf{e}_1}{r_2}\right) & m_4
\end{array}\right).
\end{equation}
By the equality between the elements of the first row in the above expression one deduces that $\rho_1^{\trsp}\mathsf{e}_2$ and $\rho^{\trsp}_2\mathsf{e}_2$ must be parallel. This  proves one part of the statement in $(ii)$  and implies, in particular, that $\rho_1\rho_2^{\trsp}={\rm diag}(m_1,m_4)$ with $m_1,m_4\in \{\pm 1\}$.  In order to conclude the proof of $(ii)$ and in the same time prove $(i)$ , we need to show that $[\gamma]$ is a line segment parallel to $\rho_1^{\trsp}\mathsf e_2$ (and to $\rho_2^{\trsp}\mathsf e_2$).
Observe that $\rho_1\rho_2^{\trsp}={\rm diag}(m_1,m_4)$ implies  $\rho_2^{\trsp}\mathsf e_1=m_1\rho_1^{\trsp}\mathsf e_1$ , so that the equation \eqref{nabla_y} simplifies to 
\begin{equation}\label{gamma_e2}
\left(\begin{array}{cc}
-m_1\sin\left(\frac{\gamma(t)\cdot \rho_1^{\trsp}\mathsf{e}_1}{r_1}\right) & 0 \\ m_1\cos\left(\frac{\gamma(t)\cdot \rho_1^{\trsp}\mathsf{e}_1}{r_1}\right) & 0 \\ 0 & m_4
\end{array}\right)=\left(\begin{array}{cc}
-m_4\sin\left(\frac{\gamma(t)\cdot \rho_1^{\trsp}\mathsf{e}_1}{r_2}\right) & 0 \\ m_1\cos\left(\frac{\gamma(t)\cdot \rho_1^{\trsp}\mathsf{e}_1}{r_2}\right) & 0\\ 0 & m_4
\end{array}\right),
\end{equation}
for every $t\in [0,1]$.
By differentiating the above equality
restricted to the first elements of the first and second rows
one gets 
\begin{equation}
\label{gamma_e2_derivation}
\begin{aligned}
m_1\cos\left(\frac{\gamma(t)\cdot 
\rho_1^{\trsp}\mathsf{e}_1}{r_1}\right)\frac{\dot{\gamma}(t)\cdot \rho_1^{\trsp}\mathsf{e}_1}{r_1}
\,&=\,
m_4\cos\left(\frac{\gamma(t)\cdot \rho_1^{\trsp}\mathsf{e}_1}{r_2}\right)\frac{\dot{\gamma}(t)\cdot \rho_1^{\trsp}\mathsf{e}_1}{r_2}\\
\sin\left(\frac{\gamma(t)\cdot 
\rho_1^{\trsp}\mathsf{e}_1}{r_1}\right)\frac{\dot{\gamma}(t)\cdot \rho_1^{\trsp}\mathsf{e}_1}{r_1}
\,&=\,
\sin\left(\frac{\gamma(t)\cdot \rho_1^{\trsp}\mathsf{e}_1}{r_2}\right)\frac{\dot{\gamma}(t)\cdot \rho_1^{\trsp}\mathsf{e}_1}{r_2}
\end{aligned}
\end{equation}
 It turns out that  \eqref{gamma_e2} and  \eqref{gamma_e2_derivation} can be satisfied   only if 
\begin{equation}
\label{gamma_line_segment}
\dot{\gamma}(t)\cdot \big(\rho_1^{\trsp}\mathsf e_1\big)=0,\qquad\mbox{ for every }t\in [0,1],
\end{equation} 
 which implies that $[\gamma]$ is a line segment parallel to $\rho_1^{\trsp}\mathsf e_2$  (thus accordingly also to $\rho_2^{\trsp}\mathsf e_2$).  
To prove previous assertion, we distinguish two cases:
\begin{itemize}
\item if $r_1\neq r_2$,  call $s:=m_1/m_4$ and fix $t\in [0,1]$. Condition \eqref{gamma_e2} grants that $\gamma(t)\cdot \rho_1^{\trsp}{\sf e}_1/r_1$ and $s\gamma(t)\cdot \rho_1^{\trsp}{\sf e}_1/r_2$ have the same sine and cosine. Then, since sine and cosine cannot simultaneously vanish, \eqref{gamma_e2_derivation} yields $\dot{\gamma}(t)\cdot \rho_1^{\trsp}{\sf e}_1/r_1=s\dot{\gamma}(t)\cdot \rho_1^{\trsp}{\sf e}_1/r_2$, whence necessarily $\dot\gamma(t)\cdot \rho_1^{\trsp}{\sf e}_1=0$.
\item if $r_1=r_2$, by hypotheses we have that $\det\rho_1=-\det\rho_2$. Since $\rho_1\rho_2^{\trsp}={\rm diag}(m_1,m_4)$, we conclude that $m_1m_4=-1$, or equivalently $m_1=-m_4$. Now the first condition in \eqref{gamma_e2_derivation} gives 
\[ \frac{d}{dt}\,\sin\left(\frac{\gamma(t)\cdot 
\rho_1^{\trsp}\mathsf{e}_1}{r_1}\right)=\cos\left(\frac{\gamma(t)\cdot 
\rho_1^{\trsp}\mathsf{e}_1}{r_1}\right)\frac{\dot{\gamma}(t)\cdot \rho_1^{\trsp}\mathsf{e}_1}{r_1}=0,\]
so that the map $t\mapsto \gamma(t)\cdot \rho_1^{\trsp}{\sf e}_1/r_1$ is constant and accordingly that $\dot{\gamma}(t)\cdot \rho_1^{\trsp}\mathsf{e}_1=0$ for every $t\in [0,1]$.
\end{itemize}
 
This concludes the proof of the 
necessary condition of the lemma in the case where 
$\gamma(0)=(0,0)$. 

Considering now the case $v:=\gamma(0)-(0,0)\neq0$, 
define $\hat{\omega}:=\omega-v$ and $\hat y_k:=y_k\circ\tau_v$, $k=1,2$ (recall form Section \ref{notation} that $\tau_{v}:=\cdot + v\in {\rm Trs}(2)$). By Remark \ref{trsrot}, one can easily verify that 
\begin{equation}
\hat y_k\,=\,T_{u_k}\circ R_k\circ \hat R_{\theta_k}\circ C_{r_k}\circ \rho_k,
\end{equation}
where 
$\theta_k:=(w_k\cdot\mathsf e_1)/r_k$ and $u_k:=v_k + R_k \circ C_{r_k}(w_k)$, with $w_k:=\rho_k\big(\gamma(0)-(0,0)\big)$, for $k=1,2$.
Observe that  the domain $\hat {\omega}$ is partitioned into $\hat\omega_1$ and $\hat\omega_2$ by the subdivision curve $[\gamma]-v$ which satisfies the condition $\gamma(0)-v=(0,0)$. It is now clear that $y\in \Wiso(\omega)$ implies $\hat y:=y\circ\tau_{v}=\hat y_1\chi_{\hat{\omega}_1}+\hat y_2\chi_{\hat{\omega}_2}\in \Wiso(\hat{\omega})$, which further implies that 
$[\gamma]-v$ (and hence $[\gamma]$) is a line segment parallel to $\rho_1^{\trsp}\mathsf e_2$ and $\rho_2^{\trsp}\mathsf e_2$, implying $\rho_1\rho_2^{\trsp}={\rm diag}(\sigma_1,\sigma_2)$, for some $\sigma_1,\sigma_2\in \{\pm 1\}$, and that 
\[
v_1 + R_1 \circ C_{r_1}(w_1)=v_2+ R_2 \circ C_{r_2}(w_2)\quad\mbox{ and }\quad \big(R_1\hat R_{\theta_1}\big)^{\trsp}
\big(R_2\hat R_{\theta_2}\big)
=\,
{\rm diag}
\big(\sigma_1\,\sigma_2,\sigma_1,\sigma_2\big),
\]
which are precisely conditions $(i)$, $(ii)$ and $(iii)$.

\textbf{(Sufficiency)} 
Let $y_1,y_2\in {\rm Cyl}$ satisfy conditions $(ii)$ and $(iii)$. Let $v:=\gamma(0)-(0,0)$ and let $\rho\in \SO(2)$ be a rotation which brings the line segment $[\gamma]-v$ to the vertical position. Let $\ubar y_k:=y_k\circ \tau_{v}\circ \rho^{\trsp}$. By denoting $u:=v_1+R_1\circ C_{r_1} (w_1)$ and $R:=R_1\circ  \hat R_{\theta_1}$ we have by $(iii)$ that
$\ubar y:=\ubar y_1\chi_{\omega_1}+\ubar y_2\chi_{\omega_2}$ is of the form
\begin{equation}\label{explicit_energy_minimizer}
\ubar {y}(x_1,x_2)=\left\{\begin{aligned}
&T_u R\Big(r_1 \big(\cos(x_1/r_1)-1\big),\sigma^1_1 r_1 \sin(x_1/r_1), \sigma^1_2 x_2\Big)^{\trsp},\ && x_1\leq 0,\\
&T_u R\Big(\sigma_1\sigma_2\, r_2 \big(\cos(x_1/r_2)-1\big),\sigma^1_1 r_2 \sin(x_1/r_2), \sigma^1_2 x_2\Big)^{\trsp},\ &&x_1>0,
\end{aligned}\right.
\end{equation}
where $\sigma_k^1\in \{\pm 1\}$ are such that $\rho_1\rho^{\trsp}={\rm diag}(\sigma_1^1,\sigma_2^1)$ (which follows form the fact that $\rho_1^{\trsp}\mathsf e_2\parallel [\gamma]$).
By construction, 
$\bar y\in C^1(\ubar{\omega}, \mathbb{R}^3)$ with $\ubar \omega=\rho(\omega-v)$.
Simple computations give $\partial_1\ubar y, \partial_2\ubar y\in \WW^{1,2}(\ubar{\omega}, \mathbb{R}^3)$, which implies that $\ubar y\in \WW^{2,2}(\ubar{\omega}, \mathbb{R}^3)$.
Note also that $\nabla \ubar y(x')^{\trsp}\,\nabla \ubar y(x')=\mathbb{I}_3$ for a.e.\ $x'\in \ubar{\omega}$. Therefore $\ubar y\in \Wiso(\ubar{\omega})$, thus accordingly $y:=\ubar y\circ\rho\circ\tau_{-v}\in \Wiso(\omega)$.
\end{proof}
\begin{remark}
\label{iso_piecewise_rmk}
{\rm 
Observe that the condition 
``${\rm det}\rho_1=-{\rm det}\rho_2$ whenever $r_1=r_2$''
permits to exclude the trivial case where
we patch together pieces of  cylinders $y_1$ and $y_2$ having the same curvatures (i.e. $\det \rho_1/r_1=\det\rho_2/r_2$, according to formula \eqref{cyl}). 
Clearly, this case does not force any condition
on $[\gamma]$. 

Moreover, an argument similar to that in the proof 
of Lemma \ref{iso_piecewise}
allows to prove necessary and sufficient conditions
for having $y\in \Wiso(\omega)$ of the form $y=y_1\chi_{\omega_1}+y_2\chi_{\omega_2}$ 
with, say, $y_2$ affine (using our terminology,
a cylinder with $r_2=+\infty$). In this case, condition $(i)$
remains the same
and condition $(ii)$ 
reduces to
$\rho_1^{\trsp}\mathsf e_2\parallel[\gamma]$ (while $\rho_2\in {\rm Orth}(2)$ can be arbitrarily chosen). 
Moreover, for a chosen $\rho_2\in {\rm Orth}(2)$, 
condition $(iii)$ becomes
\[
(R_1 \hat R_{\theta_1})^{\trsp}R_2 \hat R_{\theta_2}=\left(\begin{array}{cccc}
{\rm det}(\rho_1\rho_2^{\trsp}) &\temp & 0& 0\\ \cline{1-4}
\begin{array}{c}
0\\
0
\end{array} &\temp & \quad \rho_1\rho_2^{\trsp}\\
\end{array}\right)\quad\mbox{ and }\quad v_1+R_1C_{r_1}(w_1)
=\,
v_2+R_2C_{r_2}(w_2)
\]
with $w_k:=\rho_k\big(\gamma(0)-(0,0)\big)$ and $\theta_k:=w_k\cdot\mathsf e_1/r_k$.
}
\end{remark}

\begin{remark}[Properties of ``roto-translations"]\label{trsrot}
{\rm 
The following two properties, regarding
the composition of
cylinders, translations and rotations,
can be easily proved.
\noindent
\begin{itemize}
\item [(i)]
Fix $R\in \SO(3)$ and $T_w\in {\rm Trs}(3)$. Then  $R\circ T_w=T_{Rw}\circ R$. 

\item [(ii)]
Let $\tau_v\in {\rm Trs}(2)$
and $\hat R_\theta\in \SO(3)$ be defined by 
\[\hat R_\theta
:=
\left(\begin{array}{ccc}
\cos\theta & -\sin\theta & 0\\
\sin\theta & \cos\theta & 0\\
0 & 0 & 1
\end{array}\right).
\]
Then $C_{r}\circ\tau_v 
= T_{C_r(v)}
\circ
\hat R_{(v\cdot\mathsf e_1)/r}
\circ C_{r}$,
for every positive real number $r$. 
\end{itemize}
In particular, property (ii) justifies
the choice of the representation
used for the elements in ${\rm Cyl}$
and it is useful for
the proof of Lemma \ref{iso_piecewise}.
}
\end{remark}
Given a piecewise constant 
$\overline{A}$
and referring to Lemma \ref{minimizers_x'}
(see also the discussion after its statement),
we set 
\begin{equation}\label{rk}
\mathfrak{r}_k:=\left\{
\begin{aligned}
& \frac{ a_k(1+2\beta)}{1+\beta},\  &&
\mbox{ if}\quad b_k=a_k,\\
& \frac{ a_k}{1+\beta},\ && 
\mbox{ if}\quad b_k=-a_k,\\
& a_k+\frac{ b_k\,\beta}{1+\beta},\  && 
\mbox{ if}\quad |a_k|>|b_k|,\\
& b_k+\frac{ a_k\beta}{1+\beta},\  && 
\mbox{ if}\quad |b_k|>|a_k|,
\end{aligned}
\right.
\quad\mbox{ for every } k=1,\ldots,n.
\end{equation}
Recall that $\{0,\pm\mathfrak r_k\}$
are the eigenvalues 
(principal curvatures)
of the (constant)
curvature tensors 
ranging in $\mathcal{N}_k$.

\begin{thm}\label{general_cond}
Let $\overline{A}$ be of the form \eqref{Abar}. Assume that $\mathfrak{r}_k\neq \mathfrak{r}_j$ for all $1\leq k<j\leq n$ such that $\mathcal{H}^1\big(\partial\omega_k\cap \partial \omega_j\big)>0$. 
Then there exists a pointwise minimizer $y\in \Wiso(\omega)$ of $\mathcal{E}^0$ if and only if the following conditions are satisfied:
\begin{itemize}
\item [$(a)$] $[\gamma_k]$ is a line segment with 
$\gamma_k(0), \gamma_k(1)\in \partial
\omega$,
for every $k=1,\ldots,n-1$;
\item [$(b)$] $\gamma_k\big((0,1)\big)\cap\gamma_j\big((0,1)\big)=\emptyset$ for all $k\neq j=1,\ldots,n-1$;
\item [$(c)$] every \emph{non flat} region $\omega_k$, i.e. $\omega_k$ with 
corresponding $\mathfrak{r}_k\neq 0$, satisfies:
$\partial\omega_k\cap\omega$ consists of connected components which are orthogonal to some eigenvector 
(principal curvature direction)
of the matrices 
of $\mathcal N_k$ corresponding
to $\mathfrak r_k$.
\end{itemize} 
\end{thm}
\begin{proof}
 
The sufficiency part of the statement follows by straightforward computations, as in the proof of Lemma \ref{iso_piecewise}. In order to prove necessity, we focus on the case  $n=2$, when $\omega$ is subdivided into two Lipschitz subdomains $\omega_1$ and $\omega_2$ by a curve $\gamma:=\gamma_1$ as in Definition \ref{subdivision}, since the general case can be achieved by an induction argument as a consequence of our definition of Lipschitz subdivision of the domain $\omega$.

Let $y\in \Wiso(\omega)$ be a pointwise minimizer of $\mathcal E^0$. Note that on both subdomains $\omega_1$ and $\omega_2$ the target curvature tensor $\overline A$ is constant. Then by the definition of pointwise minimizers, by Lemma \ref{min_A_const} and Lemma \ref{minimizers_x'} we deduce that $y=y_1\chi_{\omega_1}+y_2\chi_{\omega_2}$, with
$
y_k=T_{v_k}\circ R_k\circ C_{1/|\mathfrak{r}_k|}\circ \rho_k\in {\rm Cyl}$, $k=1,2,
$
where $\mathfrak{r}_k$ is given by \eqref{rk} and $\rho_k$ is such that ${\rm A}_{y_k}\equiv(\det \rho_k)\,\rho_k^{\trsp}\,{\rm diag}\big(|\mathfrak{r}_k|,0\big)\rho_k\in \mathcal{N}_k$. Since $\mathfrak{r}_1\neq \mathfrak{r}_2$, by Lemma \ref{iso_piecewise} and Remark \ref{iso_piecewise_rmk} we obtain that $[\gamma]$ must be a line segment and that $\rho_k^{\trsp}\mathsf e_2$ must be parallel to $[\gamma]$ (or equivalently that the eigenvector $\rho_k^{\trsp}{\sf e}_1$ of ${\rm A}_{y_k}$ is orthogonal to $[\gamma]$) whenever $\mathfrak{r}_k\neq 0$, $k=1,2$, which is precisely the statement of $(a)$ and $(c)$ in the case in which $n=2$.
 
\end{proof}
\begin{remark}
{\upshape
Let $k$ and $j$ be such that $\mathcal H(\partial\omega_k\cap\partial\omega_j)>0$. Observe that when $\mathfrak{r}_k=\mathfrak r_j$ (this may happen, though $\overline A_k\neq \overline A_j$), 
this condition does not impose 
that  $\partial\omega_k\cap \partial\omega_j$  is a line segment. Indeed, when $\mathfrak r_k=\mathfrak r_j$, 
a pointwise minimizer $y$, 
when restricted to $\omega_k$ and 
$\omega_j$, will be given  by some cylinders $y_k$ and $y_j$ with $r_k=1/|\mathfrak r_k|$ and $r_j=1/|\mathfrak r_j|$, respectively, which have the same curvatures $\det \rho_k |\mathfrak{r}_k|=\mathfrak r_k=\mathfrak r_j=\det \rho_j |\mathfrak{r}_j|$.
This fact, as observed in Remark \ref{iso_piecewise_rmk},  does not impose any 
further conditions on
$\partial\omega_k\cap \partial\omega_j$.
}
\end{remark}
 Note that,  if the target curvature does not induce
any flat region, 
the presence of a pointwise minimizer
forces the subdivision lines $[\gamma_k]$
to be all parallel
(see Figure \ref{parallel_min}, (A) and (B)).
When instead a flat region is present
in the subdivision, 
this can give rise to a pointwise minimizer,
even if the $[\gamma_k]$ 
are not mutually parallel 
(see Figure \ref{parallel_min}, (C) and (D)).
Finally, observe that in this
case a subdomain of type $(iii)$
and $(iv)$ can coexist 
(tough they cannot be neighbors).
\begin{figure}[H]
  \begin{subfigure}[b]{0.43\textwidth}
    \includegraphics[width=0.9\textwidth]{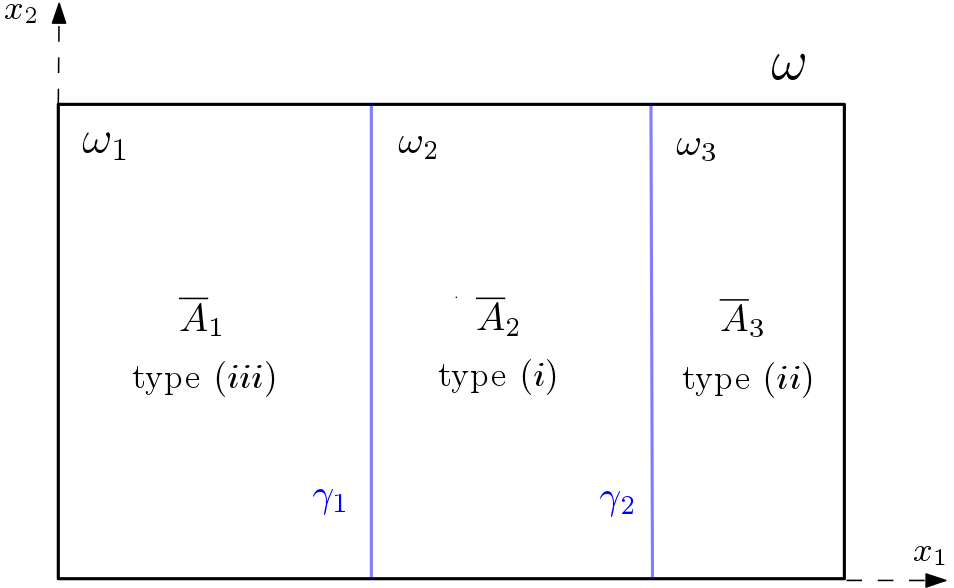}
    \caption{}
    \label{fig:f12}
  \end{subfigure}
  \begin{subfigure}[b]{0.43\textwidth}
    \includegraphics[width=0.9\textwidth]{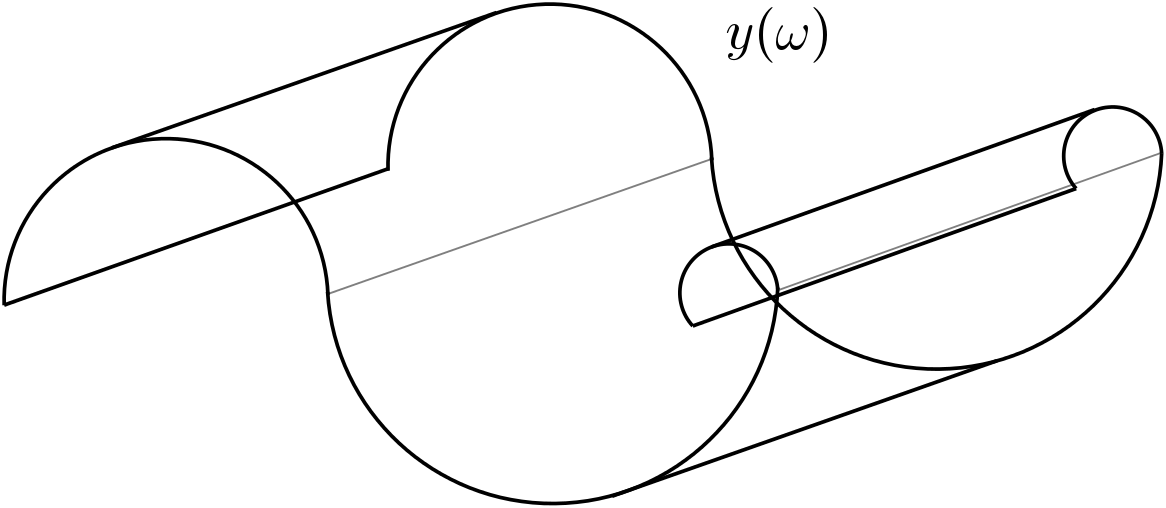}
    \caption{}
    \label{fig:f22}
  \end{subfigure}
 
 \bigskip
 
    \begin{subfigure}[b]{0.43\textwidth}
    \includegraphics[width=0.9\textwidth]{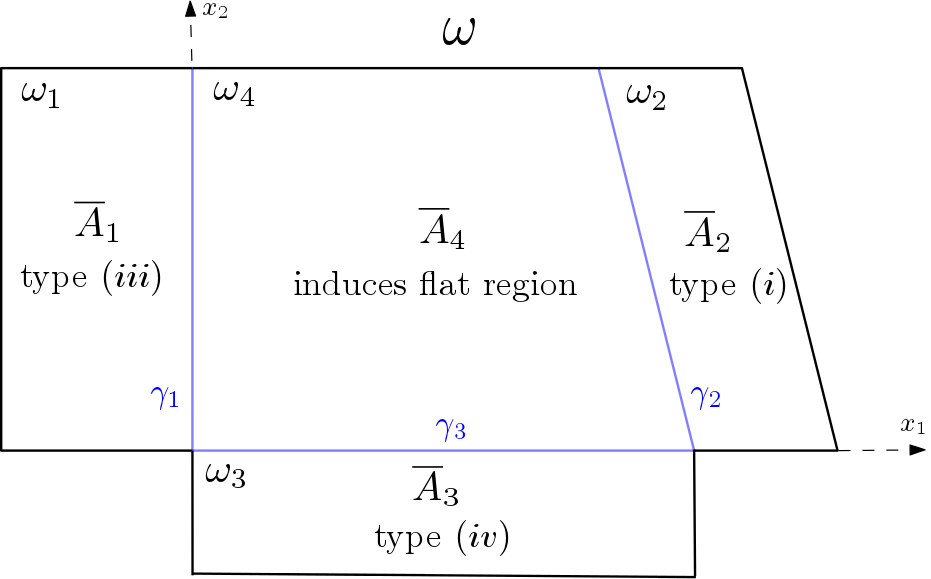}
    \caption{}
    \label{fig:f1}
  \end{subfigure}
  \begin{subfigure}[b]{0.43\textwidth}
    \includegraphics[width=0.9\textwidth]{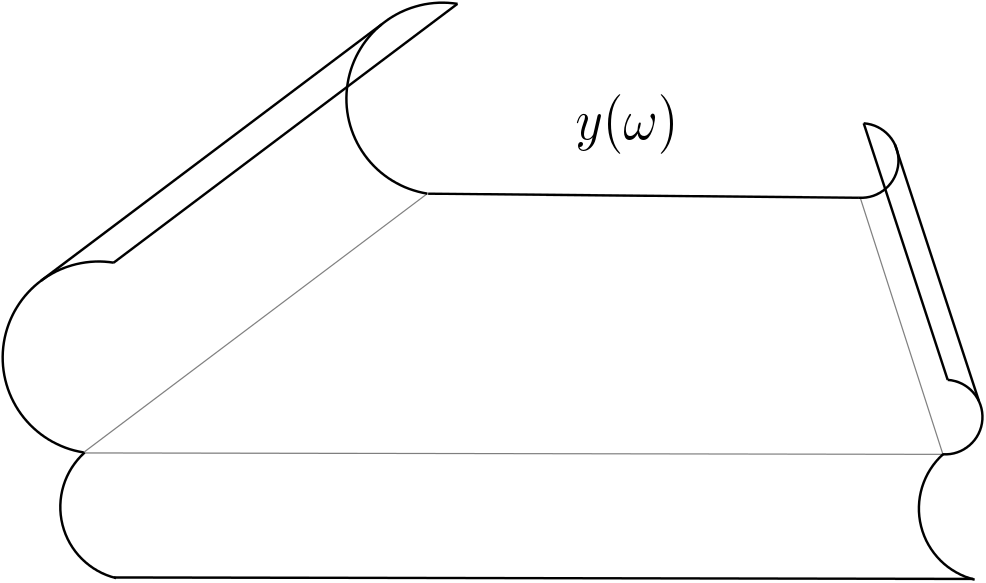}
    \caption{}
    \label{fig:f2}
  \end{subfigure}
  \caption{\label{parallel_min}Examples of reference 
domains with
given target curvature
$\overline A=\sum_{k=1}^3
\overline A_k\chi_{\omega_k}$, which guarantees the existence of a pointwise minimizer $y$ in the case when there are no flat regions induced
(figure (A)) and in the case when a flat regions are present (figure (C)).
Corresponding examples of $y(\omega)$
are illustrated in pictures (B) and (D), respectively.
}
\end{figure}
Point $(c)$ above implies that
for every $k$ and $j$ such that $\omega_k$
and $\omega_j$ are neighbor 
(i.e. share a piece of boundary,
in symbols 
$\mathcal{H}^1(\partial\omega_k\cap\partial\omega_j)>0$) 
it cannot be that 
$A_k$ is of type $(iii)$ 
(see Lemma \ref{minimizers_x'})
and 
$A_j$ is of type $(iv)$
at the same time.
This is because, if not so, 
from point $(c)$ above it would
follow that the line segment 
$[\gamma]=\partial\omega_k\cap\partial\omega_j$
is simultaneously parallel to $\mathsf e_2$
and to $\mathsf e_1$, which is absurd.
Hence, a reference domain endowed with target curvature 
as in Figure \ref{no_point_min} does not
admit a pointwise minimizer.
\begin{figure}[H]
    \includegraphics[width=0.3\textwidth]{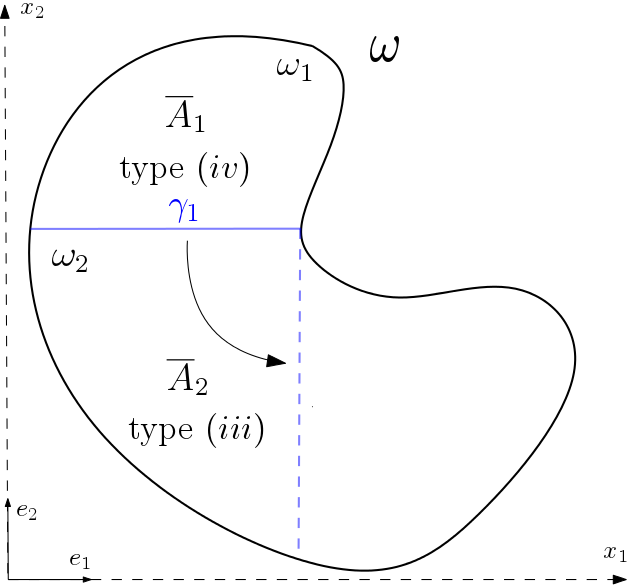}
  \caption{ An example of reference 
domain with
given target curvature
$\overline A=
\overline A_1\chi_{\omega_1}+\overline A_2\chi_{\omega_2}$ which does not allow for a pointwise minimizer 
$y$. This is because $\overline A_1$ of type $(iv)$ forces $[\gamma_1]$ to be parallel to $\mathsf e_1$, while $\overline A_2$ of type $(iii)$ forces $[\gamma_1]$ to be parallel to $\mathsf e_2$.
\label{no_point_min}}
\end{figure}
\section[Application to thin gel sheets]
{Application to thin gel sheets}
\label{gel_sheets}
In this section, we apply the reduced model derived in Section \ref{en_density} to the study of thin sheets of polymer gel. In the present context, a polymer gel is a network of cross-linked polymer chains swollen with a liquid solvent. Denote by $v$ the volume per solvent molecule, by $\overline{N}$ the density of polymer chains in the reference volume and define $\mathbb{R}^{3\times 3}_1:=\{F\in\mathbb{R}^{3\times 3}:{\rm det}F\geq 1\}$. The dimensionless free-energy density for isotropic 
 
and homogeneous
 
polymer gels is of Flory-Rehner type (see \cite{Doi09}) and is given by the 
function $ W :\mathbb{R}^{3\times 3}_1\to \mathbb{R}$ defined as
\begin{equation}\label{flory_rehner_homogeneous}
 W (F):=\frac{v\overline{N}}{2}\big(|F|^2-3\big)+W^{\chi}_{vol}({\rm det}F)+\delta({\rm det}F-1),\qquad \mbox{ for every } F\in\mathbb{R}^{3\times 3}_1.
\end{equation}
Here $\chi\in (0,1/2]$ and $\delta\geq 0$ are fixed dimensionless constants depending on the physical and chemical properties of the material and on environmental conditions, respectively. The function $W_{vol}^{\chi}:[1,+\infty)\to (-\infty,0]$ is of class $C^{\infty}$ on $(1,+\infty)$ and  (right-) continuous at $1$,  with  
$$ W_{vol}^{\chi}(1)=0,\quad \frac{d}{dt}W_{vol}^{\chi}(t)<0\,\mbox{ for every }t\in (1,+\infty)\quad\mbox{ and }\inf_{t\in [1,+\infty)}W^{\chi}_{vol}(t)=\chi-1.$$

Our attention is in particular 
focussed on a \emph{heterogeneous} thin gel sheet occupying the reference configuration $\Omega_h$. More precisely, we suppose that the sheet
is characterized by a $z$-dependent cross-linking density, which in turn determines a $z$-dependent
density $\overline N^h$ of polymer chains. 
At the same time, we suppose that $\overline N^h$
is a perturbation of a constant value $\overline N$,
namely
\begin{equation}\label{poly_chains}
\overline N^h(z):=\overline{N}+hg\left(z',\frac{z_3}{h}\right),\qquad\mbox{ for a.e.\ }z\in \Omega_h\,\mbox{ and every }\, 0<h\ll 1,
\end{equation}
where $g\in\LL^{\infty}(\Omega)$ and  
\begin{equation}
\label{media_0}
\fint_{-h/2}^{h/2}\overline N^h(z',z_3)\,\rmd z_3=\overline N,\qquad\mbox{ for a.e.\ }z'\in \omega.
\end{equation}
Observe that the condition \eqref{media_0} is equivalent to $\int_{-\nicefrac{1}{2}}^{\nicefrac{1}{2}}g(x',t)\,\rmd t=0$ for a.e.\ $x'\in \omega$.
Using the model energy density
\eqref{flory_rehner_homogeneous},
we can describe this heterogeneous system via 
the family of densities
\begin{equation}\label{flory_rehner}
\begin{aligned}
\overline W^h(z,F):=\frac{v}{2}\big(\overline{N}+hg(z',z_3/h)\big)\big(|F|^2-3\big)+W^{\chi}_{vol}({\rm det}F)+\delta({\rm det}F-1)
\end{aligned}
\end{equation}
for a.e.\ $z\in \Omega_h$, every $F\in \mathbb{R}^{3\times 3}_1$ and every $h>0$.
Letting $\{W^h\}$ be the associated family of rescaled densities $W^h:\Omega\times \RRR\to \R\cup\{+\infty\}$ defined by 
\begin{equation}
\label{rescaled_density_gel}
W^h(x,F):=\overline W^h\big((x',hx_3),F\big),\qquad\mbox{ for a.e.\ } x\in \Omega\mbox{ and every } F\in \RRR_1
\end{equation}
and declaring it to be equal $+\infty$ on $\RRR\setminus\RRR_1$, one can show 
(see the details in \cite{ADLL})
that there exist constants $\alpha>1$ and $\Theta\in \mathbb{R}\setminus \{0\}$,
depending on 
$v$, $\overline N$, $\chi$ and $\delta$,
such that for a.e.\ $x\in \Omega$ it holds
\begin{equation}\label{flory_rehner_minimum}
W^h(x,F)\,=\,\min_{\RRR}W^h(x,\cdot)\quad \mbox{ if and only if }\quad F\in \big(\alpha+hb(x)\big)\SO(3)\,\mbox{ with }\, b:=\Theta g.
\end{equation}
Moreover, one can show that $\{W^h\}$ is a family of frame indifferent functions that uniformly converges (in the sense of $(iii)$ in Definition \ref{adm_en_dens}) to $W$ and have quadratic growth. The hypothesis \eqref{media_0} ensures that the spontaneous strain $B=b\mathbb I_3$ in this case satisfies the assumption \eqref{comp_Dmin}. 

By a suitable change of variable in order to switch from the energy wells that are $h$-close to $\alpha\mathbb I_3$ to those that are $h$-close to $\mathbb I_3$ and than using the theory developed in Section \ref{en_density}, we obtain the corresponding 2D Kirchhoff model, which is in this case
given by the energy functional
\begin{equation}
 \mathcal E ^0(y):=\frac{1}{24}\int_{\omega}Q_2\big({\rm A}_y(x')-\overline A(x')\big)\,\rmd x'+\frac{1}{2}\int_{\Omega}Q_2\big(b(x)\mathbb I_2\big)\,\rmd x-\frac{1}{24}\int_{\omega}Q_2\big(\overline A(x')\big)\,\rmd x',
\end{equation}
for every  $y\in \WW^{2,2}(\omega, \R^3)$ satisfying $(\nabla y)^{\trsp}\nabla y=\alpha^2\mathbb I_2$ a.e.\ in $\omega$ (that we will briefly call an $\alpha$-isometry), and $+\infty$ otherwise in $\WW^{1,2}(\Omega, \R^3)$. The relation with the initial 3D model can be seen trough the target curvature tensor $\overline A$, given by
\[
\overline A=12\int_{-1/2}^{1/2}x_3\, b(\cdot,x_3)\,\rmd x_3\,\mathbb I_2,\quad\mbox{ a.e.\ in }\omega,\qquad b=\Theta g,
\] 
and through the quadratic form $Q_2$ defined by 
\[
Q_2(F):=\min_{d\in \R^3}D^2W(\alpha\mathbb I_3)[\hat F+ d\otimes {\sf f}_3]^2,\qquad\mbox{ for every }F\in \R^{2\times 2},
\] 
and explicitly reads as
\begin{equation}\label{gelQ2}
Q_2(F)=2G|F_{\rm sym}|^2+\Lambda(\alpha)\,\tr^2 F, \qquad\mbox{ for every }F\in \R^{2\times 2},
\end{equation}
where $G$ and $\Lambda(\alpha)$ are positive constants depending only on the (fixed) physical properties of the material.
 
\begin{remark}\label{gel_motivation}
{\rm 
We remark that the (rescaled) energy densities $ W^h(x,\cdot) $ defined by \eqref{flory_rehner} are minimized 
on
\[ \big(\alpha+hb(x)\big) \SO(3),\qquad \mbox{
for every } h>0 \mbox{ and a.e.\ } x\in \Omega,\]
for some $\alpha>1$.
Moreover, they uniformly converge to $ W $ given by \eqref{flory_rehner_homogeneous}, which is minimized at $\alpha\,\SO(3)$. However, 
by directly confronting
formulas \eqref{flory_rehner_homogeneous} and \eqref{flory_rehner},
one can check that the densities $ W^h $
cannot be rewritten in the ``prestretch''
form 
\[
 W^h(x,F)=
W\bigg(\bigg(1+h\frac{b(x)}{\alpha}\bigg)^{-1}F\bigg)
 .
\]
}
\end{remark}
\begin{example}\label{example_2}
{\rm 
Consider a thin film made of polymeric gel occupying the domain $\Omega_h$ where $\omega=(-d,d)\times (0,\ell)$ and with associated family of energy densities $ \{\overline W^h\} $ given by \eqref{flory_rehner}. Suppose that the variation of the number of polymeric chains $\overline N^h$ given by \eqref{poly_chains} is such that the associated  function  $g$  is of the form
 
\[
g(x',x_3):=\left\{\begin{aligned} g_1(x_3),\ & \mbox{ if } x'\in (-d,0]\times(0,\ell)\\
g_2(x_3),\ & \mbox{ if } x'\in (0,d)\times (0,\ell),
\end{aligned}\right.
\]
 
with $g_1,g_2\in \LL^{\infty}(-1/2,1/2)$, satisfying 
$\int_{\nicefrac{-1}{2}}^{\nicefrac{1}{2}}g_1(t)\,\rmd t=\int_{\nicefrac{-1}{2}}^{\nicefrac{1}{2}}g_2(t)\,\rmd t=0$, and
\[a_1:=12\int_{\nicefrac{-1}{2}}^{\nicefrac{1}{2}}x_3\,\Theta g_1(x_3)\,\rmd x_3\neq 12\int_{\nicefrac{-1}{2}}^{\nicefrac{1}{2}}x_3\,\Theta g_2(x_3)\,\rmd x_3=:a_2\] with $a_1$, $a_2$ non zero .
 
In turn, the limiting 2D model is characterized by the target curvature tensor $\overline A$ that equals $a(x')\mathbb{I}_2$ at each $x'\in \omega$, where
$a(x')=a_1$ if $x'\in (-d,0]\times(0,\ell)$ and $a(x')=a_2$ if $x'\in (0,d)\times(0,\ell)$.
 
Finally, using the results of Lemma \ref{minimizers_x'} and  Theorem \ref{general_cond} , we can determine the minimizers of the limiting energy  $\mathcal{E}^0$ . Note that we are in the case of Lipschitz $2$-subdivision of $\omega$ into subdomains $\omega_1:=(-d,0)\times(0,\ell)$ and $\omega_2:=(0,d)\times(0,\ell)$. Given that the subdivision curve is $[\gamma_1]=\partial\omega_1\cap\partial\omega_2=[0,\ell]$ (a line segment parallel to $\mathsf{e}_2$) and $\overline A$ is of type $(i)$ (see Lemma \ref{minimizers_x'}), a pointwise minimizer of  $\mathcal{E}^0$  exists and is any (up to rotations and translations in $\mathbb{R}^3$)  $\alpha$-isometry $y:=y_1\chi_{\omega_1}+y_2\chi_{\omega_2}$, with $y_1,y_2$ given  for every $(x_1,x_2)\in \omega$  by
\begin{equation}\label{minimum_energy_configurations}
\begin{aligned}
&y_1(x_1,x_2):=\alpha\,\Big(\alpha r_1 \Big(\cos\big(x_1/(\alpha r_1)\big)-1\Big),\sigma_1 \alpha r_1 \sin\big(x_1/(\alpha r_1)\big), \sigma_2 x_2\Big)^{\trsp},\\
&y_2(x_1,x_2):=\alpha\,\Big(\sigma_0 \alpha r_2 \Big(\cos\big(x_1/(\alpha r_2)\big)-1\Big),\sigma_1 \alpha r_2 \sin\big(x_1/(\alpha r_2)\big), \sigma_2 x_2\Big)^{\trsp},
\end{aligned}
\end{equation}
with 
\[r_k:=\frac{1}{|\mathfrak{r}_k|}\quad\mbox{ and }\quad \mathfrak{r}_k=a_k\frac{2G+2\Lambda(\alpha)}{2G+\Lambda(\alpha)}, \quad k=1,2,\] (according to Lemma \ref{minimizers_x'} and \eqref{gelQ2})
  and appropriate choice (depending on the sign of $\mathfrak{r}_k$, $k=1,2$) of $\sigma_i\in \{-1,1\}, i=0,1,2$.

Since the pull-back of the second fundamental form associated with $y_1(\omega_1)$ and $y_2(\omega_2)$ respectively, is given by
 $${\rm A}_{y_1}=\sigma_1\sigma_2\left(\begin{matrix}
|\mathfrak{r}_1| & 0\\
0 & 0
\end{matrix}\right)\quad\mbox{ and }\quad {\rm A}_{y_2}=\sigma_0\sigma_1\sigma_2\left(\begin{matrix}
|\mathfrak{r}_2| & 0\\
0 & 0
\end{matrix}\right),$$
 
it is clear that there exists two different (up to rotations and translation in $\mathbb{R}^3$) minimizing surfaces $y(\omega)$. The choice of $\sigma_1\in \{-1,1\}$ determines one of the two possible options for $y_1$, represented by a dashed or a full black line in Figure \ref{fig:f13} below. For any chosen value of $\sigma_1$, the values of $\sigma_2$ and $\sigma_0$ are immediately determined by the sign of $\mathfrak{r}_1$ and $\mathfrak{r}_2$, respectively.

\begin{figure}[H]
	\centering
    \includegraphics[width=0.6\textwidth]{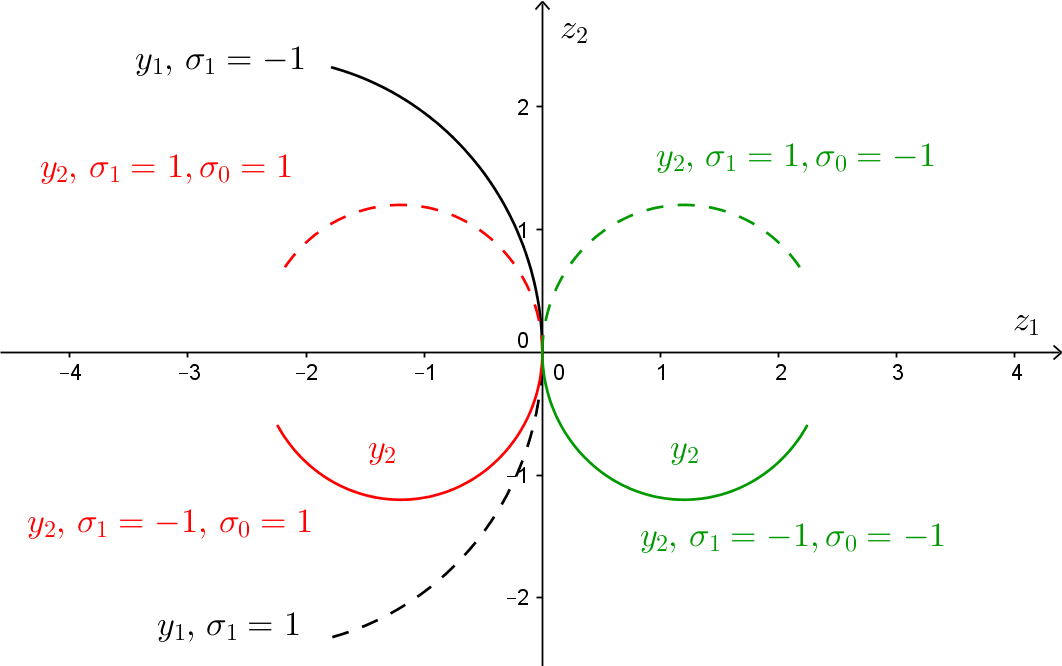}
\caption{Intersection with $(z_1,z_2)$-plane in $\mathbb{R}^3$ of two possible (up to roto-translations) minimizing surfaces $y(\omega)$. One corresponds to a full line (by choosing $\sigma_1=-1$) and the other one to a dashed line (by choosing $\sigma_1=1$). For both choices of $\sigma_1$, the value of the target curvature $\mathfrak{r}_2$ uniquely determines the value of $\sigma_0$ and thus ``decides" whether (both) intersections are black-red (if $\sigma_0=1$) or black-green (if $\sigma_0=-1$) lines.}
\label{fig:f13}
\end{figure}  
}
\end{example}

\section{Conclusions}

In this paper, we have considered a
family of 3D energy functionals that is relevant from the viewpoint of applications to shape morphing materials, especially in the context of swelling gels. 
As remarked in the Introduction, the starting 3D model
\eqref{Bh}--\eqref{gr_cond_wh}
(which reduces to
\eqref{flory_rehner}--\eqref{flory_rehner_minimum}
in the specific gel case), which may be employed to accurately describe
the swelling of polymer gels, is characterized by spontaneous stretches
but not in general representable in the ``pre-stretch form''. Another peculiarity of such a family of 3D energies is that the spontaneous stretches are naturally related to the elastic parameters of the material. Hence, heterogeneities in the stiffness can be exploited to program the target shape of the system.

Having in mind applications to free-swelling, thin gel sheets with heterogeneous stiffness, we have derived by dimension reduction from the aforementioned 3D model a Kirchhoff plate theory (Sections \ref{en_density}--\ref{sec_min}). This plate model, whose governing equations are \eqref{energia_limite}--\eqref{adt_limit_energy}, is then specialized to thin gel sheets in Section \ref{gel_sheets}. A central result of the theory is the expression of the spontaneous curvature as a function of parameters that can be traced back to the three-dimensional stiffness field.
The derivation of the
limiting model
is restricted to the case where the compatibility condition
\eqref{b_check} is fulfilled by the spontaneous strain.
As explained in the introduction,
this fact allows us to 
perform a rigorous
dimension reduction
with standard arguments.
Even though	it is possible to realize experimentally simple systems that fulfill
such condition,
this paper raises and
leaves open a mathematically relevant problem, that is, finding the general limiting Kirchhoff
model without the restriction 
\eqref{b_check}, whose complete solution would potentially
give new insight into the dimension reduction from 3D elasticity to plate theory.

We have then
investigated the pointwise minimizers of
the 2D model,
restricting the attention to the case 
where the target curvature is piecewise constant
(see Figure \ref{fig: folding}, Figure \ref{parallel_min} (B) 
and (D), for some sketches
of the configurations
which occur in this case).
The interest in this special class of minimizers is twofold. On the one hand, such a class corresponds to some of the simplest structures that can be realized experimentally, which yet can find interesting engineering applications (i.e. foldable structures, see the forthcoming \cite{ADLL}). On the other hand,
it opens the way to the study of a huge class of 
open minimum problems 
(that is the minimization of
the functional 
\eqref{energia_limite}--\eqref{def:target_curvature}
in $\Wiso(\omega)$, given
an arbitrary bounded
$\overline A:
\omega\to\Sym(2)$),
for which ready-made analytical
tools are not
yet available.

\subsection*{Acknowledgements}
This work has been
funded by the European Research Council through 
the ERC Advanced Grant 340685-MicroMotility.
We thank A. DeSimone for helpful discussions.  We would also like to thank the anonymous 
referees for their careful reading of our paper, 
and for several constructive remarks.


\bibliographystyle{siam}


\end{document}